\documentclass[12pt, letterpaper]{amsart}
\usepackage{color}
\newcommand{\indentalign}{\hspace{0.3in}&\hspace{-0.3in}}
\newcommand{\la}{\langle}
\newcommand{\ra}{\rangle}

\renewcommand{\Im}{\operatorname{Im}}

\newcommand{\R}{\mathbb{R}}
\newcommand{\C}{\mathbb{C}}

\usepackage{amssymb}
\usepackage{amsthm}
\usepackage{amsxtra}
\usepackage{graphicx}

\newtheorem{theorem}{Theorem}

\newtheorem{proposition}[theorem]{Proposition}
\newtheorem{lemma}[theorem]{Lemma}

\theoremstyle{remark}

\numberwithin{equation}{section}

\numberwithin{theorem}{section}

\numberwithin{table}{section}

\numberwithin{figure}{section}

\ifx\pdfoutput\undefined
  \DeclareGraphicsExtensions{.pstex, .eps}
\else
  \ifx\pdfoutput\relax
    \DeclareGraphicsExtensions{.pstex, .eps}
  \else
    \ifnum\pdfoutput>0
      \DeclareGraphicsExtensions{.pdf}
    \else
      \DeclareGraphicsExtensions{.pstex, .eps}
    \fi
  \fi
\fi

\begin{document}

\title{Scattering for the $L^2$ supercritical point NLS}

\author[Riccardo Adami]{Riccardo Adami$^{\scriptsize 1}$}
\author[Reika Fukuizumi]{Reika Fukuizumi$^{\scriptsize 2}$}
\author[Justin Holmer]{Justin Holmer$^{\scriptsize 3}$}



\maketitle
\begin{center} 
$^1$ DISMA, Politecnico di Torino, Italy; \\
\email{riccardo.adami@polito.it}
\end{center}

\begin{center} 
$^2$ Graduate School of Information Sciences, Tohoku University, Japan; \\
\email{fukuizumi@math.is.tohoku.ac.jp}
\end{center}

\begin{center} 
$^3$ Department of Mathematics, Brown University, USA; \\
\email{holmer@math.brown.edu}
\end{center}

\begin{abstract}
We consider the 1D nonlinear Schr\"odinger equation with focusing point nonlinearity. 
``Point'' means that the pure-power nonlinearity has an inhomogeneous potential and the potential 
is the delta function supported at the origin. This equation is used to model a Kerr-type medium with 
a narrow strip in the optic fibre. There are several mathematical studies on this equation and 
the local/global existence of solution, blow-up occurrence and blow-up profile have been investigated.    
In this paper we focus on the asymptotic behavior of the global solution, i.e, 
we show that the global solution scatters as $t\to \pm \infty$ in the $L^2$ supercritical case. 
The main argument we use is due to Kenig-Merle, but 
it is required to make use of an appropriate function space (not Strichartz space) 
according to the smoothing properties of the associated integral equation. 
\end{abstract}

\section{Introduction} 
In this paper, we address a theoretical study on a model, proposed in \cite{MB}, that describes   
a wave propagation in a 1D linear medium containing a narrow strip of nonlinear material, where 
the nonlinear strip is assumed to be much smaller than the typical wavelength. 
Considering such nonlinear strip may allow to model a wave propagation in nanodevices, in particular 
the authors in \cite{HTMG} consider some nonlinear quasi periodic super lattices and 
investigate an interplay between the nonlinearity and the quasi periodicity.  
Such a strip is described as an impurity, i.e. a delta measure 
in the nonlinearity of nonlinear Schr\"odinger equation. 
For applications in nanodevices, it should be important to study NLS 
with a quasi periodic location of delta measures, but in this paper, 
as a first step, we will treat 
the Schr\"odinger equation which has only one impurity in the nonlinearity:  
\begin{equation}
\left\{
\label{E:pNLS}
\begin{array}{ll}
i\partial_t \psi + \partial_x^2 \psi + K(x) |\psi|^{p-1}\psi =0, & t\in \R,~ x\in \R \\
\psi(x,0)=\psi_0(x) & 
\end{array}
\right.
\end{equation}
where $p>1$, and $K=\delta$, $\delta$ is the Dirac mass at $x=0.$ 
This singularity in the nonlinearity is interpreted as the linear Schr\"odinger equation:
\begin{equation*}
i\partial_t \psi + \partial_x^2 \psi  =0,  \qquad t\in \R,  \quad x \ne 0 
\end{equation*}
together with the jump condition at $x=0$
\begin{eqnarray*}
&&\psi(0,t) := \psi (0-, t)= \psi(0+, t) \\ 
&&\partial_x \psi (0+, t) - \partial_x \psi (0-,t)=-|\psi(0,t)|^{p-1} \psi(0,t). 
\end{eqnarray*}
Remark that this equation (\ref{E:pNLS}) also appears 
as a limiting case of nonlinear Schr\"odinger equation 
with a concentrated nonlinearity (see \cite{CFNT}). 
\vspace{3mm}

In \cite{AT, HL1}, it was proved that the equation (\ref{E:pNLS})  
is locally well-posed for any $\psi_0 \in H^1 (\mathbb{R})$ for $p>1$,  
and Equation (\ref{E:pNLS}) has two conservative quantities: 
the mass
$$M(\psi) = \int |\psi|^2$$
and the energy
$$E(\psi) = \frac12 \int |\partial_x \psi|^2 - \frac{1}{p+1} |\psi(0)|^{p+1}.$$
The mass condition for the global existence/blow-up, further an analysis of the blow-up profile were  
established in \cite{HL1, HL2}. Furthermore, 
the problem of asymptotic stability of the standing waves of equation (\ref{E:pNLS})  
has been treated in \cite{BKKS} and \cite{KKS}.  

As far as we know, the asymptotic behavior, in particular, 
the scattering of the solution is not known for (\ref{E:pNLS}). 
For the standard NLS, i.e. $K\equiv 1$, in one dimensional case, 
such a result in $H^1$ was firstly established in \cite{N}. 
This topic has been very active these decades  
thanks to a breakthrough result by Kenig-Merle \cite{KM}. 
Our proof therefore essentially will be based 
on Kenig-Merle \cite{KM}, and some results after  
\cite{KM}, for example \cite{HR}. However, it is required to
make use of an appropriate function space (not Strichartz space) according to the
smoothing properties of the associated integral equation to (\ref{E:pNLS}).

Higher-dimensional models with a generalization of the delta potential have been introduced in \cite{ADFT} 
and in \cite{CCT} for the three and two-dimensional setting, respectively. 
While, at a qualitative level, the model in dimension three behaves like that in dimension one, 
the two-dimensional setting displays some uncommon features 
still to be understood (for the analysis of the blow-up, see \cite{ACCT}).
\vspace{3mm}

We remark that the model of a NLS with a standard power nonlinearity and a linear point interaction
has been studied in \cite{BV}.
\vspace{3mm}

{\bf Notation.} If $I$ is an interval of $\R$, 
%$E$ is a Banach space,
and $1\le r\le \infty$, then $L^r_I$ is the space of
strongly Lebesgue measurable, complex-valued functions $v$ from $I$ into $\C$ satisfying 
$\|v\|_{L^r_I}:= \int_{I}|v(t)|^r dt <+\infty$ if $r<+\infty$, when $r=+\infty$, $\|v\|_{L^{\infty}_I}:=\sup_{t\in I} |v(t)|<+\infty$.
The space $C_I^0 E$ denotes the space of continuous functions on $I$ with values in a Banach space $E$. 

For $s\in \R,$ we define the Sobolev space   
\begin{equation*}
{H}^{s}=
\{v \in \mathcal{S}'(\R),~ \|v\|_{{H}^s}:=\|(1+|\xi|^2)^{\frac{s}{2}} \widehat{v}(\xi)\|_{L^2_{\R}} <+\infty \},
\end{equation*}
and the homogeneous Sobolev space 
\begin{equation*}
\dot{H}^{s}=
\{v \in \mathcal{S}'(\R),~ \|v\|_{\dot{H}^{s}}:=\||\xi|^s \widehat{v}(\xi)\|_{L^2_{\R}} <+\infty \},
\end{equation*}
where $\widehat f$ is the Fourier transform of the function $f$. Thus, $H^0=\dot{H}^0=L^2_{\R},$ and this will be simply denoted 
as $L^2$.
Sometimes we put an index $t$ or $x$ 
like $\dot{H}^s_t$ or $\dot{H}^s_x$ to enlighten which variable concerns.   
For $\alpha \in \R$, $|\nabla|^{\alpha}$ 
denotes the Fourier multiplier with symbol $|\xi|^{\alpha}.$
For $s\ge 0$, define $v\in H^s_I$ if, when $v(x)$ is extended to $\tilde{v}(x)$ on $\R$ by setting
$\tilde{v}(x) = 0$ for $x \notin I$, then $\tilde{v} \in  H^s$; in this case we set $\|v \|_{H^s_I} = \| \tilde{v} \|_{H^s}.$
Finally, $\chi_I$ denotes the characteristic function for the interval $I \subset \R$.
\vspace{3mm}

The equation (\ref{E:pNLS}) has a scaling invariance: if $\psi(x,t)$ is a solution to (\ref{E:pNLS}) 
then $\lambda^{\frac{1}{p-1}} \psi(\lambda x, \lambda^2 t)$, $\lambda>0$ is also. The scale-invariant Sobolev 
space for (\ref{E:pNLS}) is $\dot{H}^{\sigma_c}$ with 
$$\sigma_c = \frac12 - \frac1{p-1},$$
thus, for (\ref{E:pNLS}), $p=3$ is the $L^2$ critical setting. 
If $p>3$, then $0<\sigma_c< \frac12$ and
$$\frac14<\frac{2\sigma_c+1}{4}<\frac12, \quad 
-\frac14<\frac{2\sigma_c-1}{4}<0.$$ 
We take $q$ and $\tilde q$ to be given by
$$\frac{1}{q} = \frac12 - \frac{2\sigma_c+1}{4}, \quad 
\frac{1}{2}=\frac{1}{\tilde q} - \frac{1-2\sigma_c}{4},$$
and from the definition of $\sigma_c$, we find that
$$q = 2(p-1) \,, \qquad \tilde q = \frac{2(p-1)}{p}.$$
In the remainder of the paper, once $p>3$ is selected, 
we will take $\sigma_c$, $q$ and $\tilde q$ to have the corresponding values as defined above.  
\vspace{3mm}

Recall that by Sobolev embedding, one has
$$\| \psi \|_{L_{\R}^q} \lesssim \| \psi \|_{\dot H^{\frac{2\sigma_c+1}{4}}}, \qquad
\| f \|_{\dot H^{\frac{2\sigma_c-1}{4}}} \lesssim \| f \|_{L_{\R}^{\tilde q}}.$$
More generally than the above case, $\sigma_c$ should satisfy $-\frac{1}{2} \le \sigma_c<\frac{1}{2}$ to apply this Sobolev embedding, 
that is, the case $\sigma_c=0$ (namely $p=3$) is included for this embedding.  
\vspace{3mm}

First, we recall here the local wellposedness result of (\ref{E:pNLS}) established in Theorem 1.1 of \cite{HL1}. 

\begin{proposition} Let $p>1$ and $\psi_0 \in H^1.$ Then, there exist
$T^*>0$ and a solution $\psi(x,t)$ to (\ref{E:pNLS}) on $[0,T^*)$ satisfying for $T<T^*,$ 
\begin{eqnarray*}
&& \psi \in C^0_{[0,T]} H^1_x  \cap C^0_{\R} H_{(0,T)}^{\frac{3}{4}}, \\
&& \partial_x \psi \in C^0_{\R_x \setminus \{0\}} H_{(0,T)}^{\frac{1}{4}}.  
\end{eqnarray*}
Here, the derivatives $\partial_x \psi(0^{\pm}, t):= \lim_{x \to \pm 0 } \partial_x \psi (x,t)$, 
exist in the sense of $H^{\frac{1}{4}}_{(0,T)}$
and $\psi$ satisfies 
$$ \partial_x \psi(0^+, t)-\partial_x \psi(0^-, t)=-|\psi(0,t)|^{p-1} \psi(0,t)$$ 
as an equality of $H_{(0,T)}^{\frac{1}{4}}$ functions (not pointwisely in $t$). 

Among all solutions satisfying the above regularity conditions, it is unique. 
Moreover, the data-to-solution map $\psi_0 \mapsto \psi$, as a map $H_x^1 \to C^0_{[0,T]} H_x^1$, is continuous, and 
if $T^*<+\infty,$ then $\lim_{t \uparrow T^*} \|\partial_x \psi(t)\|_{L^2_{\R}}=+\infty.$ 
\end{proposition}
\vspace{3mm}

Hereafter, the solution to (\ref{E:pNLS}) satisfying the above regularity condition will be referred to as 
$H^1_x$ solution to (\ref{E:pNLS}). 

\vspace{3mm}

The local virial identity has been also proved in \cite{HL1}.  
For any smooth weight function $a(x)$ satisfying  $a(0)=\partial_x a(0)=\partial_x^{(3)} a(0)=0$, 
the solution $\psi$ to \eqref{E:pNLS} satisfies
\begin{equation}
\label{E:local-virial}
\partial_t^2 \int a(x) |\psi|^2 \, dx = 4 \int \partial_x^{(2)} a |\partial_x \psi|^2 
- 2 \partial_x^{(2)} a(0)|\psi(0)|^{p+1}-\int \partial_{x}^{(4)} a |\psi|^2.
\end{equation}

\begin{proposition}[{\cite[Prop 1.3]{HL1}} sharp Gagliardo-Nirenberg inequality]
\label{P:sharpGN}
For any $\psi \in H^1$, 
\begin{equation}
\label{E:1-126}
|\psi(0)|^2 \leq \| \psi \|_{L^2} \|\partial_x \psi \|_{L^2} \,.
\end{equation} Equality is achieved if and 
only if there exist $\theta\in \mathbb{R}$, $\alpha>0$ and $\beta>0$  such that  $\psi(x)=\alpha e^{i\theta}\varphi_0(\beta x)$, 
where $\varphi_0=2^{\frac{1}{p-1}} e^{-|x|}$ is the ground state solution to (\ref{E:pNLS}) (see \cite{HL1}).
\end{proposition}

\begin{theorem}[{\cite[Prop 1.4]{HL1}} $L^2$ supercritical global existence/blow-up dichotomy]
\label{T:global-blowup}
Suppose that $\psi(t)$ is an $H_x^1$ solution of \eqref{E:pNLS} for $p>3$ satisfying
\begin{equation} \label{E:ME}
M(\psi_0)^{\frac{1-\sigma_c}{\sigma_c}} E(\psi_0)< M(\varphi_0)^{\frac{1-\sigma_c}{\sigma_c}} E(\varphi_0).
\end{equation}
Let
$$\eta(t) 
= \frac{ \|\psi\|_{L^2}^{\frac{1-\sigma_c}{\sigma_c}} \|\partial_x \psi (t) \|_{L^2}}{\| \varphi_0 \|_{L^2}^{\frac{1-\sigma_c}{\sigma_c}} 
\| \partial_x \varphi_0 \|_{L^2}}$$
Then
\begin{enumerate}
\item If $\eta(0)<1$, then the solution $\psi(t)$ is global in both time directions and $\eta(t)<1$ for all $t\in \mathbb{R}$.
\item If $\eta(0)>1$, then the solution $\psi(t)$ blows-up 
in the negative time direction at some $T_-<0$, blows-up in the positive time direction at some $T_+>0$, and $\eta(t)>1$ for all $t\in (T_-,T_+)$.   
\end{enumerate}
\end{theorem}
\vspace{3mm}

Remark that if $E(\psi_0) < 0$, then the condition (\ref{E:ME}) is satisfied, and in that case $\eta(t)>1$ is forced by (\ref{E:1-126}), 
so the condition (2) applies giving the blow-up.   
\vspace{3mm}

Main result of this paper is the following. 

\begin{theorem}(asymptotic completeness) \label{T:main} Let $p>3.$ 
Let $\psi_0 \in H^1$ and let $\psi(t)$  be a $H^1_x$ solution of (\ref{E:pNLS}) satisfying   
$$ M(\psi_0)^{\frac{1-\sigma_c}{\sigma_c}} E(\psi_0)< M(\varphi_0)^{\frac{1-\sigma_c}{\sigma_c}} E(\varphi_0)$$
and 
$$
\|\psi_0\|_{L^2}^{\frac{1-\sigma_c}{\sigma_c}} \|\partial_x \psi_0\|_{L^2} < \|\varphi_0 \|_{L^2}^{\frac{1-\sigma_c}{\sigma_c}}  
\| \partial_x \varphi_0 \|_{L^2}. $$
Then, there exist $\psi^{+}, \psi^{-} \in H^1$ such that  
\begin{equation*}
\label{conv_in_energy}
\lim_{t \to \pm \infty} \|e^{-it\partial_x^2} \psi (t)- \psi^{\pm}\|_{H^1_x}=0.
\end{equation*}
\end{theorem}

We only consider the focusing nonlinearity, but the scattering for the defocusing case is similarly proved.  

This paper is organized as follows: Below in Section 2, 
we will discuss the local theory, scattering criterion and long-time perturbation theory. 
Section 2 includes some preliminary and important results which reflect the smoothing properties of the equation (\ref{E:pNLS}). 
We will give in Section 3 the profile decomposition in $H^1$ in a form well-adapted to our equation. 
In Section 4, the asymptotic completeness in $H^1$ will be established using the results in Sections 2 and 3.  
We sometimes denote all through the paper by
$C_{\theta,...}$ a constant which depends on $\theta$ and so on.
\vspace{3mm}

\section{Local theory, scattering criterion, and long-time perturbation theory}

%We introduce notations for the wellposedness space: for any $\sigma \in \R,$
%$\dot X^\sigma$ consists of $\psi$ such that $\psi 
%\in C_t^0 \dot H_x^\sigma \cap C_x^0 \dot H_t^{\frac{2\sigma+1}{4}}$ 
%and $\psi_x \in C_{x\neq 0} \dot H_t^\frac{2\sigma-1}{4}$ 
%with finite left and right limits at $x=0$. For a $T>0$, we denote 
%$C^0_{[0,T]} \dot{H}^\sigma (I) \cap C^0 (I) \dot H_{(0,T)}^{\frac{2\sigma+1}{4}}$ by $\dot{X}^{\sigma}(I\times [0,T])$.

Write the equation (\ref{E:pNLS}) in the Duhamel form:
\begin{eqnarray} \nonumber
\psi(x,t)&=& e^{it\partial_x^2}\psi_0 +i \int_{0}^t e^{i(t-s)\partial_x^2} \delta(x) |\psi(x, s)|^{p-1} \psi(x,s)ds \\ \label{E:Duhamel}
&=& e^{it\partial_x^2}\psi_0 +i \int_{0}^t \frac{e^{\frac{ix^2}{4(t-s)}}}{\sqrt{4\pi i(t-s)}} |\psi(0,s)|^{p-1} \psi(0,s)ds. 
\end{eqnarray}
We remark that the equation (\ref{E:pNLS}) is completely solved once the one-variable
complex function $\psi(0, \cdot)$ is known: indeed, specializing \eqref{E:Duhamel} to the value $x = 0$, 
one obtains a closed, nonlinear, integral, a Volterra-Abel type equation for 
$\psi(0, \cdot)$;
%Such equation is called the {\em charge equation} and
\begin{equation} \label{E:charge}
\psi(0,t)=[e^{it\partial_x^2}\psi_0](0) +i \int_{0}^t \frac{1}{\sqrt{4\pi i(t-s)}} |\psi(0,s)|^{p-1} \psi(0,s)ds.
\end{equation} 

Now, for any $\sigma \in \R$, we define for $f \in \dot{H}^{\sigma},$ $t,s \in \R$ with $t\ge s,$
$$[\mathcal{L}_s f](x, t):=\int_{s}^t \frac{e^{\frac{ix^2}{4(t-\tau)}}}{\sqrt{4\pi i(t-\tau)}} f(\tau) d\tau.$$
Similarly, we define, for $t\in \R$, 
$$[\Lambda f](x,t):=\int_t^\infty \frac{e^{\frac{ix^2}{4(t-\tau)}}}{\sqrt{4 \pi i(t-\tau)}} f(\tau) d\tau.$$

The following smoothing properties of $\mathcal{L}_s$ and $\Lambda$ will play important roles in what follows.   
\begin{proposition} \label{P:smoothing} Let $\sigma\in \R.$   
\begin{itemize}
\item[(1)] $\|[e^{i(t-s)\partial_x^2}f](0)\|_{\dot{H}^{\frac{2\sigma+1}{4}}_t} \lesssim \|f\|_{\dot{H}^{\sigma}}$, 
for any  $f \in \dot{H}^{\sigma}$ and $t,s \in \R$.
\item[(2)] Assume $-\frac{1}{2}< \frac{2\sigma-1}{4}<\frac{1}{2}.$ Let $f \in \dot{H}^{\frac{2\sigma-1}{4}}$ and $s \in \R$.
\begin{itemize}
\item[(2a)] $\| [\mathcal{L}_s f] (0,\cdot) \|_{\dot{H}_t^{\frac{2\sigma+1}{4}}} \lesssim \|\chi_{[s,+\infty)} f\|_{\dot{H}^{\frac{2\sigma-1}{4}}}
\lesssim \|f\|_{\dot{H}^{\frac{2\sigma-1}{4}}}$ 
\item[(2b)] $\| [\Lambda f] (0,\cdot) \|_{\dot{H}_t^{\frac{2\sigma+1}{4}}} \lesssim \|f\|_{\dot{H}^{\frac{2\sigma-1}{4}}}$
\end{itemize}
\item[(3)] Assume $-\frac{1}{2}< \frac{2\sigma-1}{4}<\frac{1}{2}.$ Let $f \in \dot{H}^{\frac{2\sigma-1}{4}}$ and $s \in \R$. 
\begin{itemize}
\item[(3a)] $\|\mathcal{L}_s f\|_{L^{\infty}_{\R_t} \dot{H}_x^{\sigma}} \lesssim \|f\|_{\dot{H}^{\frac{2\sigma-1}{4}}}.$
\item[(3b)] $\|\Lambda f\|_{L^{\infty}_{\R_t} \dot{H}_x^{\sigma}} \lesssim \|f\|_{\dot{H}^{\frac{2\sigma-1}{4}}}.$ 
\end{itemize}
\end{itemize}
\end{proposition}
%The domain $\R$ can be replaced by $(0,+\infty)$ or $(-\infty,0)$. 
%Remark that without the assumption $-\frac{1}{2}< \frac{2\sigma-1}{4}<\frac{1}{2},$ 
%we have, for any $t>0$, 
%$$ \mbox{(3b)'}: \|\Lambda f\|_{\dot{H}_x^{\sigma}} \lesssim \|\chi_{[t,+\infty)} f\|_{\dot{H}^{\frac{2\sigma-1}{4}}}$$ 
%for any $\sigma\in \R.$ 

For the proof of Proposition \ref{P:smoothing}, we need some preparations. 

\begin{lemma}
\label{L:sharp-trunc}
For any $-\frac12 < \mu < \frac12$, and any $t>0$, we have
\begin{equation}
\label{E:trunc2}
\| \chi_{[0,t]}(s)f(s) \|_{\dot H_s^\mu} \lesssim \|f\|_{\dot H_s^\mu}
\end{equation}
with implicit constant independent of $t$.  
\end{lemma}
\begin{proof}
First, we claim that it suffices to show 
\begin{equation}
\label{E:trunc1}
\| \chi_{[0,+\infty)}f \|_{\dot H_s^\mu} \lesssim \|f\|_{\dot H_s^\mu}
\end{equation}
Indeed, suppose that we have proved \eqref{E:trunc1}.  
Since $\chi_{[0,t]} = \chi_{[0,+\infty)} \chi_{(-\infty,t]}$, to prove \eqref{E:trunc2} we note
\begin{align*}
\| \chi_{[0,t]} f \|_{\dot H_s^\mu} &= \| \chi_{[0,+\infty)} \chi_{(-\infty,t]} f \|_{\dot H_s^\mu} \\
&\lesssim \| \chi_{(-\infty,t]} f \|_{\dot H_s^\mu} && \text{by \eqref{E:trunc1}} \\
&= \| \chi_{[0,+\infty)} \tilde f \|_{\dot H_s^\mu} 
\end{align*}
where $\tilde f(s) = f(-s+t)$.  In the last step, we have used that
$$[ \chi_{(-\infty,t]}(s) f(s) ]\;\widehat{\;}\;(\tau) = e^{-it\tau} [ \chi_{[0,\infty)}(s)f(-s+t) ] \;\widehat{\;}\;(-\tau)$$
We continue and apply \eqref{E:trunc1} to obtain
$$\| \chi_{[0,+\infty)} \tilde f \|_{\dot H_s^\mu}  \lesssim  \| \tilde f\|_{ \dot H_s^\mu} = \|f \|_{\dot H_s^\mu}$$
where, in the last step, we used that $\widehat{\tilde f}(\tau) = e^{-it\tau} \hat f(-\tau)$.  This completes the proof of \eqref{E:trunc2} assuming \eqref{E:trunc1}.

To prove \eqref{E:trunc1}, we note $\hat \chi_{[0,+\infty)}(\tau) = \operatorname{pv} \frac{1}{i\tau} + \pi \delta(\tau)$
and thus
$$ [\chi_{[0,+\infty)}f ]\;\widehat{\;}\; (\tau) = \pi (H \hat f + \hat f)$$
where $H$ denotes the Hilbert transform.  Hence
\begin{align*}
\| \chi_{[0,+\infty)} f \|_{\dot H^\mu} &= \| |\tau|^\mu [\chi_{[0,+\infty)}f ]\;\widehat{\;}\; (\tau) \|_{L_\tau^2} \\
&\lesssim \| |\tau|^\mu (H\hat f)(\tau) \|_{L^2_\tau} + \| |\tau|^\mu \hat f(\tau) \|_{L^2_\tau}
\end{align*}
Since $-\frac12<\mu<\frac12$, we can apply Corollary of Theorem 2 on  page 205 in \cite{S}, combined with (6.4) on p. 218 of \cite{S} (for $p=2$, $n=1$, $a=2\mu$) to estimate the above as
$$\| \chi_{[0,+\infty)} f \|_{\dot H^\mu} \lesssim \| |\tau|^\mu \hat f \|_{L^2_\tau} = \|f \|_{\dot H^\mu}.$$ 
\end{proof}

\begin{proof} (of Proposition \ref{P:smoothing}) (1) was already proved in Lemma 1 of \cite{AT}, but for the sake of completeness we give a proof.  
We use here the notation $\hat{}$, which means the Fourier transform in space, and $\mathcal{F}$ is in time. It suffices to show the case $s=0$. 
Since the free Schr\"odinger group is unitary in $\dot{H}_x^{\sigma}$ for any $\sigma \in \R$, 
We may write 
$$[e^{it\partial_x^2} f](0)=\int_{\R_\xi} e^{-i\xi^2 t} \hat{f}(\xi) d\xi.$$
By a change of variables this equals 
$$\int_0^{+\infty} e^{-ikt} \frac{\hat{f}(-\sqrt{k})+\hat{f}(\sqrt{k})}{2\sqrt{k}} dk.$$
Thus the Fourier transform in time gives
$$\mathcal{F} [(e^{it\partial_x^2} f)(0)](\omega) =2\pi \frac{\hat{f}(-\sqrt{\omega})+\hat{f}(\sqrt{\omega})}{2\sqrt{\omega}} \chi_{[0,+\infty)}(\omega).$$
Therefore
\begin{eqnarray*}
\|[e^{it\partial_x^2} f](0)\|_{\dot{H}^{\eta}}^2 
&=& \pi^2 \int_{\R_{\omega}} |\omega|^{2\eta-1} |\hat{f}(-\sqrt{\omega})+\hat{f}(\sqrt{\omega})|^2 \chi_{[0,+\infty)}(\omega) d\omega \\
&\le & 2\pi^2 \int_{\R_k} |k|^{4\eta-1}|\hat{f}(k)|^2 dk \\
&=& C \|f\|_{\dot{H}^{\frac{4\eta-1}{2}}},
\end{eqnarray*}
where, again we changed the variables $\pm\sqrt{\omega}=k$ in the second inequality. 
%In particular it follows from (1) that for any $f \in \dot{H}^{-\sigma}$,    
%\begin{equation} \label{E:(1)}
%\|[e^{it\partial_x^2}f](0)\|_{\dot{H}^{-\frac{2\sigma-1}{4}} (\R)} \le \|f\|_{\dot{H}_x^{-\sigma}}.
%\end{equation}
For (2a), we may write 
\begin{eqnarray*}
[\mathcal{L}_s f](0,t)&=&\int_s^t \frac{f(\tau)}{\sqrt{4\pi i(t-\tau)}} d\tau \\
&=&\frac{1}{\sqrt{4\pi i}}\int_{-\infty}^{+\infty} (t-\tau)_{+}^{-\frac{1}{2}} \chi_{[s, \infty)}(\tau) f(\tau) d\tau
=\frac{1}{\sqrt{4\pi i}}(t_+^{-\frac{1}{2}} \ast \chi_{[s,+\infty)}f)(t),
\end{eqnarray*}
where
\begin{equation*}
t_+^{-\frac{1}{2}} :=
\left\{
\begin{array}{ll}
t^{-\frac{1}{2}}, & t>0 \\
0, & t \le 0, 
\end{array}
\right.
\qquad \widehat{t_+^{-\frac{1}{2}}}(\xi) =(i\xi)^{-\frac{1}{2}} \Gamma\left(\frac{1}{2}\right). 
\end{equation*}
We operate the Fourier transform and obtain 
\begin{eqnarray*}
\widehat{[\mathcal{L}_s f](0,\cdot)}(\xi)=\frac{(i\xi)^{-\frac{1}{2}}}{\sqrt{4i}} \widehat{\chi_{[s,\infty)} f} (\xi).  
\end{eqnarray*}
It thus follows that by Lemma \ref{L:sharp-trunc}, for $-\frac{1}{2}<\frac{2\sigma-1}{4}<\frac12,$
\begin{eqnarray*}
\|[\mathcal{L}_s f](0,\cdot)\|^2_{\dot{H}^{\frac{2\sigma+1}{4}}} 
\le C \|\chi_{[s,+\infty)} f\|^2_{\dot{H}^{\frac{2\sigma-1}{4}}} \le C \|f\|^2_{\dot{H}^{\frac{2\sigma-1}{4}}}.
\end{eqnarray*}
The proof of (2b) is similar, since 
$$ [{\Lambda} f](0,t)
= \frac{-i}{\sqrt{4\pi i}} ((-t)_+^{-\frac{1}{2}} \ast f)(t).$$

For (3a), it suffices to prove that for any $g \in \dot{H}_x^{-\sigma}(\R)$ with $\|g\|_{\dot H_x^{-\sigma}} =1,$  
$$\langle \mathcal{L}_s f, g\rangle \le \|f\|_{\dot{H}_t^{\frac{2\sigma-1}{4}}}.$$
The left hand side can be estimated as follows. 
\begin{eqnarray*}
\langle \mathcal{L}_s f, g\rangle &=& \frac{1}{\sqrt{4\pi i}}
\int_{-\infty}^{+\infty} \chi_{[s,t]}(\tau) f(\tau) [e^{i (t-\tau) \partial_x^2} \bar{g}](0) d\tau \\
&\le& C \|\chi_{[s,t]} f\|_{\dot{H}^{\frac{2\sigma-1}{4}}} \|[e^{i (t-\cdot)\partial_x^2} \bar{g}](0)\|_{\dot{H}^{-\frac{2\sigma-1}{4}}}\\
&\le & C \|f\|_{\dot{H}^{\frac{2\sigma-1}{4}}} \|g\|_{\dot{H}_x^{-\sigma}}
\end{eqnarray*}
where we have used (1) with the unitary property of free Schr\"odinger group in $\dot{H}^s_x$ for any $s\in \R$, 
and Lemma \ref{L:sharp-trunc} in the last inequality. Since (3b) can be similarly proved, we omit the proof, but 
we remark that for any $\sigma \in \R,$ (that is, without the restriction $-\frac{1}{2}< \frac{2\sigma-1}{4}<\frac{1}{2}$), 
\begin{equation} \label{eq:Lambda}
\|\Lambda f\|_{\dot{H}_x^{\sigma}} \lesssim \|\chi_{[t,+\infty)} f\|_{\dot{H}^{\frac{2\sigma-1}{4}}}.
\end{equation} 
holds. 
\end{proof}
\vspace{3mm}

%\begin{proposition} \label{P:localWP}
%Let $p\geq 3$ and $\psi_0 \in  H^{\sigma_c}$. 
%There exist a time $T^*>0$ and a unique solution $\psi \in C_{[0,T]} H^{\sigma_c}$ 
%for any $T <T^*.$
%\end{proposition}
%\vspace{3mm}

From now on, we prepare some basic facts in order to prove the asymptotic completeness. 
For the sake of simplicity we will 
study the following Propositions \ref{P:smalldata}-\ref{P:longtimeper}
only in the case $t>0,$ but we can consider the negative time $t<0$ similarly. 

\newcommand{\sd}{\text{sd}}
\begin{proposition}[small data global well-posedness] \label{P:smalldata}
Let $p\ge 3$. There exists $\delta_{\sd}>0$  such that if $\psi_0\in \dot H^{\sigma_c}$ 
and $\| [e^{it\partial_x^2} \psi_0](0)\|_{L_{t>0}^q} \leq \delta_{\sd}$, 
then $\psi \in \dot{H}^{\sigma_c}$ solving \eqref{E:pNLS} is global in $\dot H^{\sigma_c}$ and
$$\| \psi(0,t) \|_{L_{t>0}^q} \leq 2 \| [e^{it\partial_x^2} \psi_0](0) \|_{L_{t>0}^q}$$
$$\| \psi(x,t) \|_{C_{[0,\infty)}^0 \dot{H}_x^{\sigma_c}} \leq 2 \| \psi_0 \|_{\dot{H}^{ \sigma_c}}.$$
\end{proposition}
\vspace{3mm}

(Note that by Proposition \ref{P:smoothing} (1) and Sobolev embedding, the smallness assumption  
$\| [e^{it\partial_x^2} \psi_0](0) \|_{L_{t>0}^q} \leq \delta_{\sd}$ is satisfied 
if $\|\psi_0\|_{\dot{H}^{\sigma_c}} \le C \delta_{\sd}$. )
\vspace{3mm}

\begin{proof}
Define a map: for a $\psi_0 \in \dot{H}^{\sigma_c}$ given, 
$$ \mathcal{T}_{\psi_0} \psi (t):=[e^{it\partial_x^2} \psi_0](0) +i [\mathcal{L}_0(|\psi|^{p-1}\psi)](t).$$
By Proposition \ref{P:smoothing} and Sobolev embedding, we have 
\begin{eqnarray*}
\| \mathcal{T}_{\psi_0} \psi\|_{L^q_{t>0}} &\le& \|[e^{it\partial_x^2} \psi_0](0)\|_{L^q_{t>0}} 
+ \|\mathcal{L}_0 (|\psi|^{p-1} \psi) (0,\cdot)\|_{L^q_{t>0}}\\
&\le& \|[e^{it\partial_x^2} \psi_0](0)\|_{L^q_{t>0}} 
+ C\|[\mathcal{L}_0 (|\psi|^{p-1} \psi)](0,\cdot)\|_{\dot{H}_{t}^{\frac{2\sigma_c+1}{4}}} \\
&\le & \|[e^{it \partial_x^2} \psi_0](0)\|_{L^q_{t>0}} + C \|\chi_{[0,\infty)}|\psi|^p\|_{\dot{H}_{t}^{\frac{2\sigma_c-1}{4}}} \\
&\le & \|[e^{it\partial_x^2} \psi_0](0)\|_{L^q_{t>0}} + C \|\psi(0,\cdot)\|^p_{L^q_{t>0}}.
\end{eqnarray*}
Let
$$B:=\{\phi \in L^q_{t>0}:~ \|\phi\|_{L^q_{t>0}} \le 2 \|[e^{it\partial_x^2} \psi_0](0) \|_{L_{t>0}^q}\}.$$
If $\| [e^{it\partial_x^2} \psi_0](0)\|_{L_{t>0}^q} \leq \delta_{\sd}$ 
then $\mathcal{T}_{\psi_0} \psi\in B$ for any $\psi \in B$, taking $\delta_{\sd}$ sufficiently small. 

The difference $\|\mathcal{T}_{\psi_0} \psi-\mathcal{T}_{\psi_0} \tilde{\psi} \|_{L^q_t}$ is similarly estimated by 
\begin{eqnarray*}
\|[ \mathcal{T}_{\psi_0} (|\psi|^{p-1} \psi-|\tilde\psi|^{p-1} \tilde\psi) ] (\cdot) \|_{L^q_{t>0}} 
\le C(\|\psi\|_{L^q_{t>0}}^{p-1}+\|\tilde\psi\|_{L^q_{t>0}}^{p-1})\|\psi-\tilde\psi\|_{L^q_{t>0}}
\end{eqnarray*}
for $\psi, \tilde\psi \in B$. Again taking $\delta_{\sd}$ sufficiently small, we conclude that 
$\mathcal{T}_{\psi_0}$ is a contraction on $B$. There thus exists a unique solution $\tilde\psi \in B$ such that 
$\mathcal{T}_{\psi_0} \tilde\psi =\tilde\psi.$

For the last inequality in the proposition, 
we use Eq. (\ref{E:Duhamel}) for the unique solution $\tilde\psi$ obtained above in $B$. 
Inserting $\tilde\psi$ as the value of $\psi(0,t)$ at time $t$ in the RHS of (\ref{E:Duhamel}), 
The values of $\psi(x,t)$ for any $x$ can be expressed as  
$$ \psi(x,t)=e^{it\partial_x^2}\psi_0 +i \int_{0}^t \frac{e^{\frac{ix^2}{4(t-s)}}}{\sqrt{4\pi i(t-s)}} |\psi(0,s)|^{p-1} \psi(0,s)ds,$$
with $\psi(0,\cdot) \in B$. 
Then, Sobolev embedding and Proposition \ref{P:smoothing} implies 
\begin{eqnarray} \nonumber
\|\psi\|_{\dot{H}^{\sigma_c}_x} &\le& \|e^{it\partial_x^2}\psi_0\|_{\dot{H}^{\sigma_c}_x}
+\|\mathcal{L}_0(|\psi|^p \psi)(\cdot,t)\|_{\dot{H}^{\sigma_c}_x} \\ \nonumber
&\le & \|e^{it\partial_x^2}\psi_0\|_{\dot{H}^{\sigma_c}_x}+ C\|\chi_{[0,t]} |\psi|^{p-1} \psi\|_{\dot{H}^{\frac{2\sigma_c-1}{4}}} \\ \nonumber
&\le & \|\psi_0\|_{\dot{H}^{\sigma_c}_x} +C\|\chi_{[0,t]} |\psi|^{p-1} \psi\|_{L^{q}_{\R}}  \\ \label{ineq:2.5}
&\le & \|\psi_0\|_{\dot{H}^{\sigma_c}_x}+ \|\psi(0,\cdot)\|_{L^q_{t>0}}^p.
\end{eqnarray}
Since $\psi(0,\cdot) \in B$ with $\| [e^{it\partial_x^2} \psi_0](0,t) \|_{L_{t>0}^q} \leq \delta_{\sd}$,  
by Sobolev embedding and Proposition \ref{P:smoothing}(1),  
$$\|\psi(0,\cdot)\|_{L^q_{t>0}}^p \le 2^p \delta_{\sd}^{p-1}\|[e^{it\partial_x^2}\psi_0](0)\|_{L^q_{t>0}} 
\le 2^p \delta_{\sd}^{p-1}\|e^{it\partial_x^2}\psi_0(0)\|_{\dot{H}^{\frac{2\sigma_c+1}{4}}_t} 
\le 2^p \delta_{\sd}^{p-1} \|\psi_0\|_{\dot{H}^{\sigma_c}_x}.$$ 
Taking $\delta_{\sd}$ sufficiently small, the RHS of (\ref{ineq:2.5}) is bounded by $2\|\psi_0\|_{\dot{H}^{\sigma_c}_x}$. 
Note that the time continuity property follows from the fundamental solution, and this concludes  
$$\| \psi(x,t) \|_{C_{[0,\infty)}^0 \dot H_x^{\sigma_c}} \leq 2 \| \psi_0 \|_{\dot{H}_x^{\sigma_c}}.$$.
\end{proof}

\begin{proposition}[scattering criterion] \label{P:scat_cri} Let $p \ge 3$.  
Suppose that $\psi_0\in H^1$ and $\psi \in H^1_x$ solving \eqref{E:pNLS} is forward global with
$$\| \psi(0,\cdot) \|_{L_{t>0}^q} < \infty$$
and with a uniform $H^1_x$ bound 
$$\sup_{t\ge 0} \|\psi(\cdot,t)\|_{H^1_x}\le B.$$
Then $\psi(t)$ scatters in $H^1_x$ as $t\nearrow +\infty$.  This means that there exists $\psi^+\in H^1_x$ such that 
$$\lim_{t\nearrow +\infty} \|\psi(t)- e^{it\partial_x^2} \psi^+\|_{H^1_x} =0.$$
\end{proposition}
\begin{proof}
Using the equation (\ref{E:Duhamel}), we may write 
\begin{eqnarray} \label{E:Duhamel_infty}
 \psi(t)- e^{it\partial_x^2} \psi^+ = -i \int_{t}^{+\infty} e^{i(t-s)\partial_x^2} \delta(x) |\psi(s)|^{p-1} \psi(s)ds,
\end{eqnarray}
 where 
$$ \psi^+ := \psi_0 + i \int_{0}^{+\infty} e^{-is\partial_x^2} \delta(x) |\psi(s)|^{p-1} \psi(s)ds.$$
Therefore, 
\begin{eqnarray*}
\|\psi(t)- e^{it\partial_x^2} \psi^+\|_{H^1_x}
&=& \|\int_{t}^{+\infty} e^{i(t-s)\partial_x^2} \delta(x) |\psi(s)|^{p-1} \psi(s)ds \|_{H^1_x} \\
&=& \|\Lambda ( |\psi|^{p-1} \psi ) (\cdot,t)\|_{H^1_x}.
\end{eqnarray*}
Thus we shall estimate $\|\Lambda(|\psi|^{p-1} \psi)(\cdot, t)\|_{L^2_x}$ and $\|\Lambda (|\psi|^{p-1} \psi) (\cdot, t)\|_{\dot{H}^1_x}.$
First, $\|\Lambda(|\psi|^{p-1} \psi) (\cdot, t)\|_{L^2_x}$ is estimated by (3b) of Proposition \ref{P:smoothing} and the Sobolev embedding as follows. 
For any $t>0$, 
\begin{eqnarray} \nonumber
\|\Lambda (|\psi|^{p-1} \psi )(\cdot, t)\|_{L^2_x} &\le& 
\|\chi_{[t +\infty)} |\psi|^{p-1} \psi \|_{\dot{H}^{-\frac{1}{4}}} \\ \nonumber
&\le& C \|\chi_{[t,+\infty)} |\psi|^{p-1} \psi\|_{L^{\tilde q}_{\R}} \\ \label{eq:scat1}
&\le& C \|\psi\|_{L^q_{(t,+\infty)}}^p.
\end{eqnarray}
Second, by the Sobolev embedding and fractional chain rule \cite{CW}, for any $t>0$,
\begin{eqnarray} \nonumber
\|\Lambda (|\psi|^{p-1} \psi)(\cdot,t)\|_{\dot{H}^1_x} 
&\le& C \| \chi_{[t,+\infty)}|\psi|^{p-1}\psi\|_{\dot{H}^{\frac{1}{4}}_t} \\ \label{eq:scat2}
&\le& C\|\chi_{[t,+\infty)}|\psi|^{p-1}\|_{L^{r_1}_{\R_t}} \||\nabla|^{\frac{1}{4}} \chi_{[t,+\infty)}\psi\|_{L^{r_2}_{\R_t}}
\end{eqnarray}
with $\frac{1}{2}=\frac{1}{r_1}+\frac{1}{r_2},$ $1< r_1, r_2 <+\infty$. Taking $q < r_1 <+\infty$ and $2 < r_2 <4$,  by interpolation,
\begin{eqnarray*}
 \|\chi_{[t,+\infty)}|\psi|^{p-1}\|_{L^{r_1}_{\R_t}} &\le& C\|\psi\|_{L^q_{(t,+\infty)}}^{\frac{q}{r_1}} \sup_{s \ge t}|\psi(0,s)|^{(1-\frac{q}{r_1})}\\ 
& \le& C\|\psi\|_{L^q_{(t,+\infty)}}^{\frac{q}{r_1}}  \sup_{s \ge t}\|\psi(s)\|_{L^{\infty}_{\R_x}}^{(1-\frac{q}{r_1})} \\
&\le & C\|\psi\|_{L^q_{(t,+\infty)}}^{\frac{q}{r_1}} \sup_{s \ge t}\|\psi(s)\|_{H^1_x}^{(1-\frac{q}{r_1})} \le C_{B} \|\psi\|_{L^q_{(t,+\infty)}}^{\frac{q}{r_1}}
 \end{eqnarray*}
where we have used the Sobolev embedding $H^1(\R_x) \subset L^{\infty}(\R_x)$. Again by interpolation 
\begin{eqnarray*}
\||\nabla|^{\frac{1}{4}} \chi_{[t,+\infty)} \psi\|_{L^{r_2}_{\R_t}} 
&\le&  \|\chi_{[t,+\infty)} \psi\|_{\dot{H}^{\frac{1}{4}}_t}^{\frac{2}{r_2}} 
\||\nabla|^{\frac{1}{4}} \chi_{[t,+\infty)}\psi\|_{L^{\infty}_{\R_t}}^{(1-\frac{2}{r_2})} \\
& \le & C\|\chi_{[t,+\infty)} \psi\|_{\dot{H}^{\frac{1}{4}}}^{\frac{2}{r_2}} 
\left(\|\chi_{[t,+\infty)}\psi\|_{\dot{H}^{\frac{1}{4}}} +\|\chi_{[t,+\infty)}\psi\|_{\dot{H}^{\frac{3}{4}}}\right)^{(1-\frac{2}{r_2})} 
\end{eqnarray*}
where we have used the Sobolev embedding $H^1(\R_t) \subset L^{\infty}(\R_t)$ in the second inequality. 
We go back to the equation (\ref{E:Duhamel_infty}), evaluating at $x=0$, to estimate 
\begin{eqnarray*}
\|\chi_{[t,+\infty)} \psi\|_{\dot{H}^{\frac{1}{4}}} &\le& \|\chi_{[t,+\infty)}[e^{it\partial_x^2} \psi^+](0)\|_{\dot{H}^{\frac{1}{4}}} 
+\|\chi_{[t,+\infty)} \Lambda (|\psi|^{p-1} \psi)(0,\cdot))\|_{\dot{H}^{\frac{1}{4}}} \\
&\le & \|\psi^+\|_{L^2_x} + \|\chi_{[t,+\infty)}|\psi|^{p-1} \psi\|_{\dot{H}^{-\frac{1}{4}}} \\
&\le & \|\psi^+\|_{L^2_x} + \|\psi\|_{L^q_{t>0}}^p,
\end{eqnarray*}
and 
\begin{eqnarray*}
\|\chi_{[t,+\infty)} \psi\|_{\dot{H}^{\frac{3}{4}}} &\le& \|\chi_{[t,+\infty)}[e^{it\partial_x^2} \psi^+](0)\|_{\dot{H}^{\frac{3}{4}}} 
+\|\chi_{[t,+\infty)}\Lambda (|\psi|^{p-1} \psi)(0,\cdot))\|_{\dot{H}^{\frac{3}{4}}} \\
&\le & \|\psi^+\|_{H^1_x} + \|\chi_{[t,+\infty)}|\psi|^{p-1} \psi\|_{\dot{H}^{\frac{1}{4}}}. 
\end{eqnarray*}
Note that we used Lemma \ref{L:sharp-trunc}, and Proposition \ref{P:smoothing} (2b). 
Plugging these results into (\ref{eq:scat2}), we see that for $t>0$ sufficiently large, 
$\| \chi_{[t,+\infty)}|\psi|^{p-1}\psi\|_{\dot{H}^{\frac{1}{4}}}$ is small. This completes the proof combining with (\ref{eq:scat1}). 
\end{proof}

\begin{proposition}[long-time perturbation theory] \label{P:longtimeper} Let $p \ge 3$. 
For each $A \gg 1$, there exists $\epsilon_0=\epsilon_0(A) \ll 1$ and $c=c(A) \gg 1$ such that the following holds.  
Let $\psi \in H_x^1$ for all $t$ solving
$$i \partial_t \psi + \partial_x^2 \psi + \delta |\psi|^{p-1}\psi =0.$$
Let $\tilde \psi \in H_x^1$ for all $t$ and suppose that there exists $e\in L_{t>0}^{\tilde q}$ such that 
$$ i\partial_t\tilde \psi + \partial_x^2 \tilde \psi + \delta (|\tilde \psi|^{p-1}\tilde \psi - e) =0.$$
If 
$$\| \tilde \psi(0,\cdot) \|_{L_{t>0}^q} \leq A \,, \quad \| e(0,\cdot) \|_{L_{t>0}^{\tilde q}}\leq \epsilon_0$$ and 
$$\| [e^{i(t-t_0)\partial_x^2} ( \psi(t_0)-\tilde \psi(t_0))](0)\|_{L^q_{t_0 \le t <\infty}} \leq \epsilon_0$$ for some $t_0 \ge 0$, 
then 
$$\| \psi(0,\cdot) \|_{L_{t>0}^q} \leq c=c(A)<\infty.$$
\end{proposition}
\begin{proof}
Put $w=\psi - \tilde\psi$. Then $w$ satisfies 
\begin{equation} \label{E:w}
i \partial_t w + \partial_x^2 w + W=0,
\end{equation}
where 
$$W=\delta (|\tilde \psi+w|^{p-1}(\tilde \psi +w)-|\tilde \psi|^{p-1}\tilde \psi + e).$$
Since $\| \tilde \psi(0,\cdot)\|_{L^q_{[t_0, +\infty)}} \leq A$, there exists a $N=N(A)$ so that the interval 
$[t_0,+\infty)$ may be divided into the sum of $N(A)$ intervals. Namely, 
$[t_0, +\infty)= \cup_{j=1}^{N(A)} I_j$ with $I_j=[t_j,t_{j+1}]$ ($j=0,1,2,..$) so that $\| \tilde \psi(0,\cdot) \|_{L^q_{I_j}} \leq \eta$ ($\eta$ is small to be determined later). 
Let $t \in I_j$. Write the equation (\ref{E:w}) in the integral form. 
\begin{equation} \label{eq:w}
w(t)=e^{i(t-t_j)\partial_x^2} w(t_j) +i \int_{t_j}^t e^{i(t-s)\partial_x^2} W(s)ds.
\end{equation}
We estimate the time $L^q$ norm of $w$ evaluated at $x=0$. 
\begin{eqnarray*}
\|w(0,\cdot)\|_{L^q_{I_j}} \le \|[e^{i(t-t_j)\partial_x^2} w(t_j)](0)\|_{L^q_{I_j}} + \left\|\int_{t_j}^t e^{i(t-s)\partial_x^2} W(s)ds |_{x=0} \right\|_{L^q_{I_j}}. 
\end{eqnarray*}
The last term can be written as, taking into account for the delta potential in $W$,  
$$\left\|\int_{t_j}^t e^{i(t-s)\partial_x^2} W(s)ds |_{x=0} \right\|_{L^q_{I_j}}  
= \|[\mathcal{L}_{t_j} (|\tilde \psi+w|^{p-1}(\tilde \psi +w)(0,\cdot)-|\tilde \psi|^{p-1}\tilde \psi(0,\cdot) + e(\cdot))](0,\cdot)\|_{L^q_{I_j}}$$ 
and then we estimate 
as follows. 
\begin{eqnarray*} 
&&\|[\mathcal{L}_{t_j} (|\tilde \psi +w|^{p-1}(\tilde \psi +w)-|\tilde \psi|^{p-1} \tilde \psi + e)](0,\cdot)\|_{L^q_{I_j}} \\
%&\le& 
%C \|\chi_{I_j}[\mathcal{L}_{t_j} (|\tilde \psi +w|^{p-1}(\tilde \psi +w)-|\tilde \psi|^{p-1}\tilde \psi + e)](0,\cdot)\|_{\dot{H}_t^{\frac{2\sigma_c+1}{4}}} \\
%&\le& 
%C \|\chi_{I_j}(\tilde \psi+w|^{p-1}(\tilde \psi +w)-|\tilde \psi|^{p-1}\tilde \psi + e)\|_{\dot{H}_t^{\frac{2\sigma_c-1}{4}}} \\
&\le & 
C \|\tilde \psi+w|^{p-1}(\tilde \psi +w)-|\tilde \psi|^{p-1}\tilde \psi\|_{L^{\tilde q}_{I_j}} + \|e\|_{L^{\tilde q}_{I_j}}\\
&\le &  
C ( \|\tilde \psi^{p-1} w (0,\cdot)\|_{L^{\tilde q}_{I_j}} + \|w^{p}(0,\cdot)\|_{L^{\tilde q}_{I_j}})+\|e\|_{L^{\tilde q}_{I_j}},\\
\end{eqnarray*}
where, in the first inequality, we have used, by density of $C_0^{\infty}(I_j) \subset L^{\tilde{q}}(I_j)$, Sobolev embedding, and Proposition \ref{P:smoothing} (2a).

The first term of RHS is estimated by H\"older inequality as follows. 
$$ \|\tilde \psi^{p-1} w(0,\cdot)\|_{L^{\tilde q}_{I_j}} \le \|\tilde \psi (0,\cdot)\|_{L^q_{I_j}}^{p-1} \|w(0,\cdot)\|_{L^q_{I_j}}.$$ 
Thus, we have 
\begin{eqnarray*}
\|w(0,\cdot)\|_{L^q_{I_j}} &\le& \|[e^{i(t-t_j)\partial_x^2} w(t_j)] (0)\|_{L^q_{I_j}} +C\eta^{p-1} \|w(0,\cdot)\|_{L^q_{I_j}}\\
&& +C\|w(0,\cdot)\|_{L^q_{I_j}}^p+C\epsilon_0. 
\end{eqnarray*}
We then obtain
\begin{eqnarray} \label{E:B}
\|w(0,\cdot)\|_{L^q_{I_j}} &\le& 2\|[e^{i(t-t_j)\partial_x^2} w(t_j)](0)\|_{L^q_{I_j}}+2C\epsilon_0,
\end{eqnarray}
provided 
$$\eta < \left(\frac{1}{2C}\right)^{\frac{1}{p-1}}$$ 
and 
\begin{eqnarray} \label{E:A}
\|[e^{i(t-t_j)\partial_x^2} w(t_j)] (0)\|_{L^q_{I_j}}+C\epsilon_0 \le \left(\frac{1}{2C}\right)^{\frac{1}{p-1}}. 
\end{eqnarray}
Now take $t=t_{j+1}$ in (\ref{eq:w}), apply $e^{i(t-t_{j+1})\partial_x^2}$ to both hands, 
$$
e^{i(t-t_{j+1})\partial_x^2} w(t_{j+1})=e^{i(t-t_j)\partial_x^2} w(t_j) +i \int_{t_j}^{t_{j+1}} e^{i(t-s)\partial_x^2} W(s)ds, 
$$
and we take $L^q(\R_t)$ norm of this equation after evaluating at $x=0$, 
\begin{eqnarray*}
\|[e^{i(t-t_{j+1})\partial_x^2} w(t_{j+1})](0)\|_{L^q_{\R_t}} 
&\le& \|[e^{i(t-t_j)\partial_x^2} w(t_j)](0)\|_{L^q_{\R_t}} +C\eta^{p-1} \|w(0,\cdot)\|_{L^q_{I_j}}\\
&& +C\|w(0,\cdot)\|_{L^q_{I_j}}^p+C\epsilon_0.
\end{eqnarray*}
Thus, by (\ref{E:B}), 
\begin{eqnarray*}
\|[e^{i(t-t_{j+1})\partial_x^2} w(t_{j+1})](0)\|_{L^q_{\R_t}} 
&\le& 2\|[e^{i(t-t_j)\partial_x^2} w(t_j)](0)\|_{L^q_{\R_t}} +2C \epsilon_0.
\end{eqnarray*}
Iterating this inequalty starting from $j=0$, we have 
$$ \|[e^{i(t-t_{j})\partial_x^2} w(t_{j})](0)\|_{L^q_{\R_t}} \le 2^{j+2} C\epsilon_0.$$ 
To satisfy (\ref{E:A}) for all $I_j$ with $0 \le j \le N-1,$ we require 
$\epsilon_0=\epsilon_0(N)$ to be sufficiently small 
such that $ 2^{N+2} C\epsilon_0 < \left(\frac{1}{2C} \right)^{\frac{1}{p-1}}$ (i.e. $\epsilon_0$ needs to be taken in terms of $A$), 
and we obtain 
$$\| \psi(0,t) \|_{L_{t>0}^q} \leq c=c(A).$$
\end{proof}

\section{Profile decomposition}

\begin{proposition}[profile decomposition] \label{P:decomposition} Let $p \ge 3.$
Suppose that $\{\psi_n\}$ is a uniformly bounded sequence in $H^1_x$.  Then for each $M$, 
there exists a subsequence of $\{\psi_n\}$, also denoted $\{\psi_n\}$ and
\begin{enumerate}
\item for each $1\leq j \leq M$, there exists a (fixed in $n$) profile $\phi^j \in H^1$
\item for each $1\leq j \leq M$, there exists a sequence (in $n$) of time shifts $t_n^j$
\item there exists a sequence (in $n$) of remainders $w_n^M(x)$ in $H^1$ such that
$$\psi_n = \sum_{j=1}^M e^{-it_n^j\partial_x^2} \phi^j + w_n^M$$
\end{enumerate}
The time sequences have a pairwise divergence property: for $1\leq i\neq j \leq M$, we have
$$\lim_{n\to \infty} |t_n^i-t_n^j| = +\infty.$$
The remainder sequence $\{w_n^M\}_n$ has the following asymptotic smallness property
$$\lim_{M\to \infty} \Big[ \lim_{n\to \infty} \| [e^{it\partial_x^2}w_n^M](0) \|_{L_{\R_t}^q} \Big] =0.$$ 
For fixed $M$ and any $0\leq \sigma_c \leq 1$, we have the asymptotic $\dot H^{\sigma_c}$ decoupling
\begin{eqnarray} \label{E:dotnorm_expansion}
\| \psi_n \|_{\dot H^{\sigma_c}}^2 = \sum_{j=1}^M \| \phi^j \|_{\dot H^{\sigma_c}}^2 + \| w_n^M \|_{\dot H^{\sigma_c}}^2 + o_n(1), 
\end{eqnarray}
also we have 
\begin{equation} \label{E:piecewise_expansion}
| \psi_n (0)|^{p+1}=\sum_{j=1}^M |[e^{-it_n^j\partial_x^2} \phi^j](0)|^{p+1}+ |w_n^M (0)|^{p+1}+ o_n(1).
\end{equation}
\end{proposition}
\begin{proof}
For $R>0$, let $\chi_R(\xi)$ be a smooth cutoff to $R^{-1}<|\xi|<R$.  
Let $A = \limsup_{n\to \infty} \|\psi_n \|_{H^1_x}$ and $B_1 = \lim_{n\to \infty} \|[e^{it\partial_x^2}\psi_n](0)\|_{L_{\R_t}^q}$. 
If $B_1=0$, the proof is done. 
Let $B_1>0$. 
Since for $0 \le \sigma_c \le 1$,
$$\int_{|\xi|<R^{-1}} |\hat \psi_n (\xi)|^2 |\xi|^{2\sigma_c} \, d\xi \leq R^{-2\sigma_c} \|\psi_n\|_{L^2}^2 \leq A^2 R^{-2\sigma_c}$$
$$\int_{|\xi|>R} |\hat \psi_n (\xi)|^2 |\xi|^{2\sigma_c} \, d\xi \leq R^{2(\sigma_c-1)} \|\psi_n \|_{\dot H^1}^2 \leq A^2 R^{2(\sigma_c-1)}.$$
 We may take a $R_1$ large enough so that $AR_1^{-\sigma_c} \le B_1/2$ and $A R_1^{\sigma_c-1} \le B_1/2$, 
specifically $R_1 = \la 2AB_1^{-1}\ra^{\max\{ \frac{1}{\sigma_c}, \frac{1}{1-\sigma_c}\}}$ so that 
$$\lim_{n\to \infty} \| [e^{it\partial_x^2} (\delta-\check{\chi}_{R_1})*\psi_n](0)\|_{L_{\R_t}^q} \leq \frac12 B_1.$$
It thus follows, using Proposition \ref{P:smoothing}(1),
\begin{eqnarray*}
\left( \frac12 B_1 \right)^q  &\leq& \lim_{n\to \infty} \| [\check \chi_{R_1} * e^{it\partial_x^2}\psi_n](0) \|_{L_{\R_t}^q}^q \\
&\leq& \lim_{n\to \infty} \| [\check \chi_{R_1} * e^{it\partial_x^2}\psi_n](0) \|_{L_{\R_t}^2}^2 
\| [\check \chi_{R_1} * e^{it\partial_x^2}\psi_n](0) \|_{L_{\R_t}^\infty}^{q-2}.
\end{eqnarray*}
For the factor $\| [\check \chi_{R_1} * e^{it\partial_x^2}\psi_n](0) \|_{L_{t>0}^2}^2 $, we use again the smoothing estimate of Proposition \ref{P:smoothing}(1) to bound by
$$\| \check\chi_{R_1}* \psi_n \|_{\dot H_x^{-1/2}}^2 \leq R_1 \| \check\chi_{R_1}* \psi_n \|_{L^2_x}^2 \leq R_1 A^2.$$
Thus, we see 
$ \lim_{n\to \infty} \| [\check \chi_{R_1} * e^{it\partial_x^2}\psi_n](0) \|_{L_{\R_t}^\infty} >(R_1 A^2)^{-\frac{1}{q-2}}(B_1/2)^{\frac{q}{q-2}}$, 
and we take a sequence $\{t_n^1\}_n$ such that   
$$[\check\chi_{R_1}*e^{it\partial_x^2}\psi_n ](0,t_n^1) = \int \check\chi_{R_1}(-y) (e^{it_n^1\partial_x^2}\psi_n)(y) \, dy,$$
and
\begin{equation} \label{E:lowerbdd}
\frac12(R_1 A^2)^{-\frac{1}{q-2}} \left(\frac{B_1}{2}\right)^{\frac{q}{q-2}} \leq \left| \int \check\chi_{R_1}(-y) e^{it_n^1\partial_x^2} \psi_n(y) \, dy \right|.
\end{equation}
Consider the sequence $\{e^{it_n^1\partial_x^2}\psi_n\}_n$, which is uniformly bounded in $H^1_x$, and  
pass to subsequence such that $e^{it_n^1\partial_x^2}\psi_n$ converges weakly in $H^1_x$ to some $\phi^1 \in H^1$.  
By Cauchy-Schwarz inequality, using that $\| \check\chi_{R_1} \|_{\dot H^{-\sigma_c}} \lesssim R_1^{\frac12-\sigma_c}$ and (\ref{E:lowerbdd}), 
$$\| \phi^1\|_{\dot H^{\sigma_c}} \geq (R_1^{\frac12-\sigma_c})^{-1} (R_1 A^2)^{-\frac{1}{q-2}} \left(\frac{B_1}{2}\right)^{\frac{q}{q-2}} \frac12.$$
Then for any $0\leq \sigma_c \leq 1$
$$\lim_{n\to \infty} \| \psi_n - e^{-it_n^1\partial_x^2} \phi^1 \|_{\dot H^{\sigma_c}}^2 = \|\psi_n\|_{\dot H^{\sigma_c}}^2 - \|\phi^1 \|_{\dot H^{\sigma_c}}^2.$$

If $|t_n^1| \to +\infty,$ since $\|[e^{-it\partial_x^2} \phi^1](0)\|_{L^q_{\R_t}} \le \|\phi^1\|_{\dot{H}_x^{\sigma_c}}$, possibly taking  
a subsequence, we have $|[e^{-i t_n^1 \partial_x^2} \phi^1](0)|^q \to 0$ as $n\to +\infty$. 
On the other hand, since $\psi_n$ is uniformly bounded in $H^1_x$, 
there is a weak limit $\tilde{\psi} \in H^1_x$ and $\psi_n(0) \to \tilde{\psi}(0)$ as $n\to \infty$ by Proposition 4.1 of \cite{HL1}.
Then, we have 
\begin{eqnarray*}
&& \lim_{n\to \infty}|[\psi_n-e^{-i t_n^1 \partial_x^2} \phi^1](0)|^{p+1} \\
&=&\lim_{n\to \infty}\{(\psi_n(0) -[e^{-it_n^1 \partial_x^2} \phi^1](0)) (\overline{\psi_n(0) -[e^{-it_n^1 \partial_x^2} \phi^1](0)})\}^{\frac{p+1}{2}} \\
&=&|\tilde{\psi}(0)|^{p+1} = \lim_{n\to \infty} (|\psi_n(0)|^{p+1}-|[e^{-i t_n^1 \partial_x^2} \phi^1](0)|^{p+1}), 
\end{eqnarray*}
i.e. 
\begin{equation} \label{E:asym_L_infty} 
\lim_{n\to \infty}[|\psi_n(0)|^{p+1}-|[e^{-i t_n^1 \partial_x^2} \phi^1](0)|^{p+1}-|w_n^1(0)|^{p+1}]=0.  
\end{equation}
If $t_n^1 \to t^*$ for some finite $t^*$, by the time continuity of free Schr\"odinger group, 
$\lim_{n\to\infty}\psi_n(0)=\tilde{\psi}(0)=[e^{-it^*\partial_x^2} \phi^1](0)$. 
Thus we may write 
\begin{eqnarray*}
\lim_{n\to \infty}|[\psi_n-e^{-i t_n^1 \partial_x^2} \phi^1](0)|^{p+1}
&=&\lim_{n\to \infty}(|\psi_n(0)|^2-|[e^{-i t_n^1 \partial_x^2} \phi^1](0)|^2)^{\frac{p+1}{2}} \\
&=& 0 = \lim_{n\to \infty} (|\psi_n(0)|^{p+1}-|[e^{-i t_n^1 \partial_x^2} \phi^1](0)|^{p+1}),
\end{eqnarray*}
which again gives (\ref{E:asym_L_infty}).

Repeat the process, keeping the same $A$ but switching to $B_2$ obtaining $R_2$ in terms of $B_2$.  
Basically this amounts to replacing $\psi_n$ by $\psi_n - e^{-it_n^1\partial_x^2}\phi^1$ and rewriting the above to obtain $t_n^2$ and $\phi^2$ where
$$ \phi^2 = \text{weak}\lim [ e^{it_n^2\partial_x^2}(\psi_n - e^{-t_n^1\partial_x^2}\phi^1)] \quad \mbox{in} \ H^1_x.$$
As a result,
\begin{align*}
\lim_{n\to \infty} \| \psi_n - e^{-it_n^1\partial_x^2} \phi^1  - e^{-it_n^2 \partial_x^2}\phi^2 \|_{\dot{H}^{\sigma_c}}^2 
&= \lim_{n\to \infty} \| \psi_n - e^{-it_n^1\partial_x^2} \phi^1\|_{\dot{H}^{\sigma_c}}  - \| \phi^2 \|_{\dot{H}^{\sigma_c}}^2 \\
&= \lim_{n\to \infty} \|\psi_n\|_{\dot{H}^{\sigma_c}}^2 - \|\phi^1 \|_{\dot{H}^{\sigma_c}}^2- \|\phi^2\|_{\dot{H}^{\sigma_c}}^2,
\end{align*}
and same for
\begin{align*}
& \lim_{n\to \infty} |[\psi_n - e^{-it_n^1\partial_x^2} \phi^1  - e^{-it_n^2 \partial_x^2}\phi^2](0)|^{p+1} \\
&= \lim_{n\to \infty} (|\psi_n(0)|^{p+1} - |[e^{-it_n^1\partial_x^2} \phi^1](0)|^{p+1} -|[e^{-it_n^2 \partial_x^2}\phi^2](0)|^{p+1}). 
\end{align*}

If $t_n^2-t_n^1$ converged to something finite (say $t^*$), then $\phi^2$ would be the weak limit of $e^{it^*\partial_x^2}[e^{it_n^1\partial_x^2}\psi_n - \phi^1]$, 
which is zero, contradicting the lower bound.  Hence $|t_n^1-t_n^2| \to \infty$ and thus
$$\la e^{-it_n^1\partial_x^2} \phi^1, e^{-it_n^2 \partial_x^2} \phi^2 \ra_{\dot{H}^{\sigma_c}}  \to 0.$$

Again repeat this process, we have 
$$
\|\phi^1\|_{\dot{H}^{\sigma_c}}^2+\|\phi^1\|_{\dot{H}^{\sigma_c}}^2+\cdots+\|\phi^M\|_{\dot{H}^{\sigma_c}}^2
+\lim_{n\to+\infty}\|w_n^M\|_{\dot{H}^{\sigma_c}}^2= \lim_{n\to+\infty}\|\psi_n\|_{\dot{H}^{\sigma_c}}^2.
$$
Let $B_{M+1}:=\lim_{n\to +\infty}\|[e^{it \partial_x^2} w_n^M](0)\|_{L^q_{\R_t}}$ and we wish to show that $B_{M+1} \to 0.$ 
Note that from the above equality and the lower bound for $\|\phi^M\|_{\dot{H}^{\sigma_c}}$, 
we obtain 
$$\sum_{M=1}^\infty R_M^{-\theta} B_M^{\frac{q}{q-2}} \le 2A^{\frac{2(q-1)}{q-2}}, \quad \theta=\frac{1}{q-2}+\frac{1}{2}-\sigma_c=\frac{1}{2(p-2)}+\frac{1}{2}-\sigma_c>0,$$
whose LHS diverges if $B_M$ does not converge to $0$. 
\end{proof}

\begin{lemma}
\label{L:lqbd}
With $w_n^M$ as defined in Proposition \ref{P:decomposition} (in particular, $w_n^0=\psi_n$), let
$$B_M = \lim_{n\to \infty} \| [e^{it\partial_x^2} w_n^{M-1}](0) \|_{L_{\R_t}^q}.$$
Then
$$ \lim_{n\to\infty} \| [e^{i(t-t_n^M) \partial_x^2} \phi^M] (0)\|_{L_{\R_t}^q} \leq 2B_M.$$
\end{lemma}
\begin{proof}
We will write the argument for $M=1$ (the general case is analogous). As in the proof of Proposition \ref{P:decomposition}, let
$$ A=\lim_{n\to \infty} \|\psi_n\|_{H^1_x}$$
and
$$R_1 = \la 2A B_1^{-1} \ra^{\max( \frac{1}{\sigma_c}, \frac{1}{1-\sigma_c})}$$
and $\chi_{R_1}(\xi)$ be a cutoff to $R_1^{-1}\leq |\xi|\leq R_1$.    As in the beginning of the proof of Proposition \ref{P:decomposition},
\begin{align*}
\indentalign \| (\delta - \check\chi_{R_1})* e^{i (t-t_n^1) \partial_x^2} \phi^1 (0) \|_{L_{\R_t}^q}^2 
\lesssim \| [(\delta- \check\chi_{R_1})* e^{it\partial_x^2} \phi^1](0) \|_{\dot H_t^{\frac{2\sigma_c+1}{4}}}^2 \\
&\lesssim \| (\delta - \check\chi_{R_1}) * \phi^1 \|_{\dot H_x^{\sigma_c}}^2 \lesssim {R_1}^{-2\sigma_c} \|\phi^1\|_{L^2}^2 + {R_1}^{-2(1-\sigma_c)} \|\phi^1 \|_{\dot H^1}^2 \\
&\leq A^2({R_1}^{-2\sigma_c}+{R_1}^{-2(1-\sigma_c)}) \leq \frac14 B_1^2
\end{align*}
This, and the similar estimates at the beginning of the proof of Proposition \ref{P:decomposition}, show that it suffices to prove
\begin{equation}
\label{E:lqlem2}
\lim_{n\to \infty} \| \check\chi_{R_1}*e^{i(t-t_n^1)\partial_x^2}\phi^1(0) \|_{L_{\R_t}^q}^2 \leq \frac14 B_1^2, 
\end{equation}
and this can be seen as follows. By the translation invariance of $L^q_{\R_t}$ norm, 
\begin{eqnarray*}
\| \check\chi_{R_1}*e^{i(t-t_n^1)\partial_x^2}\phi^1(0) \|_{L_{\R_t}^q}
&=& \| \check\chi_{R_1}*e^{it\partial_x^2}\phi^1(0) \|_{L_{\R_t}^q}
\end{eqnarray*}
and by Sobolev embedding and Proposition \ref{P:smoothing}, we have,  
\begin{eqnarray*}
\| \check\chi_{R_1}*e^{it\partial_x^2}\phi^1(0) \|_{L_{\R_t}^q} 
&\lesssim & \|\check\chi_{R_1}*e^{it\partial_x^2}\phi^1(0)\|_{\dot{H}_t^{\frac{2\sigma_c+1}{4}}} \\
&\lesssim & \|\check\chi_{R_1}*\phi^1\|_{\dot{H}^{\sigma_c}_x} \\ 
&\lesssim &  \left(A^2 R_1^{-2(1-\sigma_c)}\right)^{\frac{1}{2}} \le B_1/2.
\end{eqnarray*}
\end{proof}

\section{Minimal non scattering solution}

In this section we will prove that there exists a {\it minimal non scattering solution}. 
For this purpose we prepare the following lemma which gives additional estimates 
under the situation (1) of Theorem \ref{T:global-blowup}. We recall that $\varphi_0$ is the ground state to (\ref{E:pNLS}). 
It is known that $\varphi_0(x)=2^{\frac{1}{p-1}} e^{-|x|}$ (see (1.9) of \cite{HL1}).

\begin{lemma} 
Let $p>3$ and $\psi_0 \in H^1_x.$  Assume $(\ref{E:ME})$ and $\eta(0)<1.$ If 
$\psi $ is a $H^1_x$ solution to (\ref{E:pNLS}), then for all $t\in \R,$
\begin{equation} \label{E:Grad_Energy}
\frac{(p-1)}{2(p+1)} \|\partial_x \psi(t)\|_{L^2}^2 \le E(\psi(t)) \le \frac{1}{2} \|\partial_x \psi(t)\|_{L^2}^2.  
\end{equation}
Furthermore, 
if we take $\delta>0$ such that $M(\psi_0)^{\frac{1-\sigma_c}{\sigma_c}}E(\psi_0)\le (1-\delta)M(\varphi_0)^{\frac{1-\sigma_c}{\sigma_c}}E(\varphi_0),$ 
then there exists $c_{\delta}>0$ such that 
for all $t \in \R$, 
\begin{equation} \label{E:LB}
 4 \|\partial_x \psi\|_{L^2}^2 -2 |\psi(0,t)|^{p+1} \ge c_{\delta} \|\partial_x \psi_0\|_{L^2}^2. 
\end{equation}
\end{lemma}

\proof The upper bound of the energy in (\ref{E:Grad_Energy}) follows by the definition of Energy $E$ and the focusing nonlinearity. 
Use the sharp Gagliardo-Nirenberg inequality and $\eta(t)<1$ for the lower bound, i.e., 
\begin{eqnarray*}
E(\psi) &\ge& \frac{1}{2} \|\partial_x \psi\|_{L^2}^2 
\Big(1-\frac{1}{p+1} \|\psi\|_{L^2}^{\frac{p+1}{2}} \|\partial_x \psi\|_{L^2}^{\frac{p-3}{2}}\Big) \\
&> & 
\frac{1}{2} \|\partial_x \psi\|_{L^2}^2 
\Big(1-\frac{1}{p+1} \|\varphi_0\|_{L^2}^{\frac{p+1}{2}} \|\partial_x \varphi_0\|_{L^2}^{\frac{p-3}{2}}\Big) \\
&=& \frac{p-1}{2(p+1)} \|\partial_x \psi\|_{L^2}^2, 
\end{eqnarray*}
where we have used the fact $\|\partial_x \varphi_0\|_{L^2}=\|\varphi_0\|_{L^2}=2^{\frac{1}{p-1}}$ in the last equality (see \cite{HL1}).
Next, we show (\ref{E:LB}). We may take $\delta_1=\delta_1(\delta)>0$ such that 
\begin{equation} \label{E:4.1}
\|\psi_0\|_{L^2}^{\frac{1-\sigma_c}{\sigma_c}}\|\partial_x \psi(t)\|_{L^2} \le (1-\delta_1) 
\|\varphi_0\|_{L^2}^{\frac{1-\sigma_c}{\sigma_c}}\|\partial_x \varphi_0\|_{L^2},
\end{equation} 
for all $t\in \R.$ Let 
\begin{equation*}
h(t):=\frac{1}{\|\varphi_0\|_{L^2}^{\frac{2(1-\sigma_c)}{\sigma_c}}\|\partial_x \varphi_0\|_{L^2}^2} 
(4\|\psi_0\|_{L^2}^{\frac{2(1-\sigma_c)}{\sigma_c}}\|\partial_x \psi(t)\|_{L^2}^2
-2\|\psi_0\|_{L^2}^{\frac{2(1-\sigma_c)}{\sigma_c}} |\psi(0,t)|^{p+1}).  
\end{equation*}
By Gagliardo-Nirenberg inequality, 
$$h(t) \ge 
g \left(\frac{\|\psi_0\|_{L^2}^{\frac{(1-\sigma_c)}{\sigma_c}}\|\partial_x \psi(t)\|_{L^2}}
{\|\varphi_0\|_{L^2}^{\frac{(1-\sigma_c)}{\sigma_c}}\|\partial_x \varphi_0\|_{L^2}}\right),$$
where $g(y):=4(y^2- y^{\frac{p+1}{2}}).$ 
The inequality (\ref{E:4.1}) implies the variable $y$ of $g(y)$ is in the interval
$0 \le y \le 1-\delta_1$ and then we see that there exists a constant $c=c_{\delta_1}>0$ such that 
$g(y) \ge cy^2$ if $0 \le y \le 1-\delta_1$. \qed
\vspace{3mm}

\begin{lemma} \label{lem:ex_waveop}(Existence of wave operator) Let $p > 3$. 
Suppose $\psi^+ \in H^1_x$ and 
\begin{equation} \label{Assum:wave_op} 
\frac{1}{2} \|\psi^+\|_{L^2}^{\frac{2(1-\sigma_c)}{\sigma_c}}
\|\partial_x \psi^+\|_{L^2}^2 <M(\varphi_0)^{\frac{1-\sigma_c}{\sigma_c}}E(\varphi_0).
\end{equation}
There exists $\psi_0 \in H^1_x$ such that $\psi$ solving (\ref{E:pNLS}) with initial data $\psi_0$ is global 
in $H^1_x$, with 
$$ M(\psi)=\|\psi^+\|_{L^2}^2, \quad E(\psi)=\frac{1}{2} \|\partial_x \psi^+\|_{L^2}^2,$$
$$ \|\partial_x \psi(t)\|_{L^2} \|\psi_0\|_{L^2}^{\frac{1-\sigma_c}{\sigma_c}} 
< \|\varphi_0\|_{L^2}^{\frac{1-\sigma_c}{\sigma_c}} \|\partial_x \varphi_0\|_{L^2}$$
and
$$ \lim_{t \nearrow +\infty} \|\psi(t)- e^{it \partial_x^2} \psi^+ \|_{H^1_x}=0.$$
Moreover, if $\|[e^{it\partial_x^2} \psi^+](0)\|_{L^q_{t>0}} \le \delta_{\sd},$ then 
$$\|\psi_0\|_{\dot{H}^{\sigma_c}} \le 2 \|\psi^+\|_{\dot{H}^{\sigma_c}}, 
\quad \|\psi(0,\cdot)\|_{L^q_{t>0}} \le 2\|[e^{it\partial_x^2} \psi^+](0)\|_{L^q_{t>0}}.$$
\end{lemma}

The statement above is for the case $t>0$, but the case $t<0$ can be similarly proved.  

\begin{proof}
It suffices to solve the integral equation:
\begin{equation*}
\psi(t)=e^{it\partial_x^2} \psi^+-i\Lambda(|\psi(0)|^{p-1}\psi(0))(t)  
\end{equation*}
for $t\ge T$ with $T$ large. Since 
\begin{equation*}
\|[e^{it\partial_x^2} \psi^+ ](0)\|_{L^q_{t>0}} \lesssim \|[e^{it\partial_x^2} \psi^+ ](0)\|_{\dot{H}_t^{\frac{2\sigma_c+1}{4}}} \le \|\psi^+\|_{\dot{H}_x^{\sigma_c}},   
\end{equation*}
there exists a large $T>0$ such that $\|[e^{it\partial_x^2} \psi^+ ](0)\|_{L^q_{[T,\infty)}} \le \delta_{\sd}$. 
Thus we may solve as in the proof of Proposition \ref{P:smalldata}.  
\begin{eqnarray*}
 \|\psi(0,\cdot)\|_{L^q_{[T,+\infty)}} &\le& \|[e^{it\partial_x^2} \psi^+ ](0)\|_{L^q_{[T,\infty)}} + C\|\Lambda(|\psi(0)|^{p-1}\psi(0))(\cdot)\|_{L^q_{[T,+\infty)}} \\
 &\le&  \|[e^{it\partial_x^2} \psi^+ ](0)\|_{L^q_{[T,\infty)}} +C\|\psi(0,\cdot)\|_{L^q_{[T,+\infty)}}^p.
 \end{eqnarray*}
%Set 
%\begin{equation*}
%B=\{v \in L^q_t:~ \|v\|_{L^q_{[T,+\infty)}} \le 2 \|[e^{it\partial_x^2} \psi^+ ](0)\|_{L^q_{[T,\infty)}} \}.  
%\end{equation*}
If $T$ is sufficiently large, we have $\|\psi(0,\cdot)\|_{L^q_{[T,+\infty)}} < 2\|[e^{it\partial_x^2} \psi^+](0)\|_{L^q_{[T,+\infty)}}.$ 
Using this, similarly as in the proof of Proposition \ref{P:scat_cri}, we obtain if $t \ge T$,
\begin{eqnarray*}
\|\psi(t)-e^{it\partial_x^2} \psi^+\|_{L^2_x} \le C\|\Lambda(|\psi(0)|^{p-1} \psi(0))\|_{L^2_x}\le \|\psi(0,\cdot)\|^p_{L^q_{[T, +\infty)}} \le C\delta_{\sd}^p,  
\end{eqnarray*}
\begin{eqnarray*}
\|\psi(t)-e^{it\partial_x^2} \psi^+\|_{\dot{H}^1_x} \le C\|\chi_{[T,+\infty)} |\psi|^{p-1} \psi\|_{\dot{H}^{1/4}_t},  
\end{eqnarray*}
which are small if $T$ is sufficiently large. 
Thus, $\psi(t) -e^{it\partial_x^2} \psi^+ \to 0$ in $H^1_x$ as $t\to +\infty$. Note that $\|\partial_x e^{it\partial_x^2} \psi^+\|_{L^2_x}=\|\partial_x \psi^+\|_{L^2}$. 
On the other hand, since $[e^{it\partial_x^2} \psi^+](0)$ is uniformly bounded in $L^q_{t>0}$, 
there exists a sequence $\{t_n\}_{n} \to +\infty$ such that 
[$e^{it_n\partial_x^2} \psi^+](0) \to 0$ as $n\to +\infty$. Together with all these facts,  we have 
\begin{equation*}
E(\psi(t))=\lim_{n\to +\infty} \left\{\frac{1}{2}\|\partial_x e^{it_n \partial_x^2} \psi^+\|_{L^2_x}-\frac{1}{p+1}|e^{i t_n \partial_x^2} \psi^+(0)|^{p+1} \right\}
=\frac{1}{2}\|\partial_x \psi^+\|_{L^2_x}. 
\end{equation*}
Similarly, $M(\psi(t))=\|\psi^+\|_{L^2_x}^2$. 
It now follows from (\ref{Assum:wave_op}) that 
$$M(\psi(t))^{\frac{1-\sigma_c}{\sigma_c}}E(\psi(t))<M(\varphi_0)^{\frac{1-\sigma_c}{\sigma_c}}E(\varphi_0),$$
and 
\begin{eqnarray*}
\lim_{t\to +\infty} \|\partial_x \psi(t)\|_{L^2_x}^2 \|\psi(t)\|_{L^2_x}^{\frac{2(1-\sigma_c)}{\sigma_c}} 
&=& \lim_{t\to+\infty}\|\partial_x e^{it\partial_x^2} \psi^+\|_{L^2_x}^2 \|e^{it\partial_x^2} \psi^+\|_{L^2_x}^{\frac{2(1-\sigma_c)}{\sigma_c}} \\
&=& \|\partial_x \psi^+\|_{L^2_x}^2 \|\psi^+\|_{L^2_x}^{\frac{2(1-\sigma_c)}{\sigma_c}} \\
&< & 2 M(\varphi_0)^{\frac{1-\sigma_c}{\sigma_c}} E(\varphi_0)=\frac{p-3}{p+1}\|\partial_x \varphi_0 \|^2_{L^2_x} \|\varphi_0\|_{L^2_x}^{\frac{2(1-\sigma_c)}{\sigma_c}}
\end{eqnarray*}
We can take a large $T$ such that $\|\partial_x \psi(T)\|_{L^2_x} \|\psi(T)\|_{L^2_x}^{\frac{1-\sigma_c}{\sigma_c}} 
< \|\partial_x \varphi_0 \|_{L^2_x} \|\varphi_0\|_{L^2_x}^{\frac{1-\sigma_c}{\sigma_c}}$.
Then, applying Theorem \ref{T:global-blowup} 
we evolve $\psi(t)$ from $T$ back to the time $0$. 
\end{proof}

We are now in position to enter in the main subject of this section.  
%We start the argument by contradiction -- suppose that there exists a solution $\psi$ for which 
%$$M(\psi)^{\frac{1-\sigma_c}{\sigma_c}} E(\psi) < M(\varphi_0)^{\frac{1-\sigma_c}{\sigma_c}} E(\varphi_0)$$
%and $\psi$ does not scatter.
If the initial data $\psi_0$ to (\ref{E:pNLS}) satisfies $M(\psi_0)^{\frac{1-\sigma_c}{\sigma_c}} E(\psi_0) \le \frac{p-1}{2(p+1)}\delta_{sd}$ and $\eta(0)<1$, we have 
$$\|\psi_0\|^{2/\sigma_c}_{\dot{H}^{\sigma_c}_x(\R)} \le \|\psi_0\|_{L^2_x}^{\frac{2(1-\sigma_c)}{\sigma_c}} \|\partial_x \psi_0\|_{L^2}^2 \le 
M(\psi_0)^{\frac{1-\sigma_c}{\sigma_c}} E(\psi_0)\le \delta_{sd}, $$
and the scattering holds by the small data scattering, Proposition \ref{P:smalldata}. 
Now let $A$ be the infimum of $M(\psi)^\frac{1-\sigma_c}{\sigma_c}E(\psi)$, taken over all evolution of $\psi$ which does not scatter.  
In what follows $\mathrm{NLS}(t)\psi$ denotes the solution to (\ref{E:pNLS}) with initial data $\psi$. 
By the above argument, $0<\frac{p-1}{2(p+1)}\delta_{sd} \le A$, and moreover due to Proposition \ref{P:scat_cri}, $A$ satisfies 
\begin{enumerate}
\item For any $\psi$ such that $M(\psi)^{\frac{1-\sigma_c}{\sigma_c}} E(\psi)<A$, it holds $\|[\mathrm{NLS}(t)\psi ](0,\cdot)\|_{L^q_{\R_t}} <\infty$, 
\item For any $A'>A$, there exists a non scattering $\mathrm{NLS}(t)\psi$ for which
$$A\leq M(\psi)^{\frac{1-\sigma_c}{\sigma_c}} E(\psi) \leq A'.$$
\end{enumerate}
If $A \ge M(\varphi_0)^{\frac{1-\sigma_c}{\sigma_c}} E(\varphi_0)$, Theorem \ref{T:main} is true. We therefore proceed with the proof by assuming 
$A < M(\varphi_0)^{\frac{1-\sigma_c}{\sigma_c}} E(\varphi_0)$.

The first task is to apply the profile decomposition to show that there exists $\psi$ 
such that  $M(\psi)^{\frac{1-\sigma_c}{\sigma_c}}E(\psi)=A$ and $\mathrm{NLS}(t)\psi$ does not scatter.  
We will call such a solution a \emph{minimal non scattering solution}.
Take a sequence of initial data $\psi_{0,n}$, with $1>\eta_n(0):=\|\psi_{0,n}\|_{L^2}^{\frac{1-\sigma_c}{\sigma_c}} \|\partial_x \psi_{0,n}\|_{L^2}
/\|\varphi_{0}\|_{L^2}^{\frac{1-\sigma_c}{\sigma_c}} \|\partial_x \varphi_{0}\|_{L^2}$, each evolving to non scattering solutions, 
for which $M(\psi_{0,n})=1$, $E(\psi_{0,n}) \geq A$ and $E(\psi_{0,n}) \to A$. Apply the profile decomposition 
to $\psi_{0,n}$ which is uniformly bounded in $H^1$ to obtain, extracting a subsequence,  
\begin{eqnarray} \label{eq:1}
&&\psi_{0,n} = \sum_{j=1}^M e^{-it_n^j \partial_x^2} \phi^j +w_n^M, \\ \label{eq:2} 
&& E(\psi_{0,n}) = \sum_{j=1}^M E(e^{-it_n^j \partial_x^2} \phi^j) +E(w_n^M) +o_n(1),
\end{eqnarray}
where $M$ will be taken large later. Remark that each term in (\ref{eq:2}) is non negative by the same reason for (\ref{E:Grad_Energy}), 
using the decompositions (\ref{E:dotnorm_expansion}) and (\ref{E:piecewise_expansion}) in $\eta_n(0)<1$. 
Taking the limit $n \to \infty$ in both hand sides, 
\begin{eqnarray}\label{eq:Edecom}
\lim_{n\to \infty} \sum_{j=1}^M E(e^{-it_n^j \partial_x^2} \phi^j) \leq A
\end{eqnarray}
for all $j$. Also, by $\sigma_c=0$ in (\ref{E:dotnorm_expansion}), we have  
\begin{eqnarray} \label{eq:L2decom}
\sum_{j=1}^M M(\phi^j) + \lim_{n\to \infty} M(w_n^M) = \lim_{n\to \infty} M(\psi_{0,n})=1.
\end{eqnarray}

Here we consider two cases.  
\begin{itemize}
 \item[Case 1] There are at least two indexes $j$ such that $\phi^j$ is not zero. 
 \item[Case 2] Only one profile is non zero, i.e. without loss of generality $\phi^1 \ne 0$, and $\phi^j=0$ for all $j\ge 2.$
\end{itemize}

We begin with Case 1. 
By (\ref{eq:L2decom}), we necessarily have $0 \leq M(\phi^j) <1$ for each $j$ which, by (\ref{eq:Edecom}),  
implies that for $n$ sufficiently large 
\begin{equation}
\label{E:methresh}
M(e^{-it_n^j \partial_x^2} \phi^j)^{\frac{1-\sigma_c}{\sigma_c}} E(e^{-it_n^j \partial_x^2} \phi^j)  \leq A_j,
\end{equation}
with each $A_j < A$.  For a given $j$, there are two possibilities. Case a) $|t_n^j| \to \infty$ as $n\to \infty$ and 
Case b) there is a finite limit $t_*$ such that $t_n^j \to t_*$ as $n\to \infty$. 
Both cases allow us to ensure the existence of a new profile $\tilde{\phi^j} \in H^1$ associated to $\phi^j$ such that 
$$\|\mathrm{NLS}(-t_n^j)\tilde{\phi}^j- e^{-it_n^j\partial_x^2} \phi^j\|_{H^1} \to 0, \quad n\to \infty;$$ 
indeed, 
if Case a) occurs, by the uniform $L^q$ integrability in time of $[e^{-it\partial_x^2} \phi^j](0)$ 
(cf. the same argument in Proposition \ref{P:decomposition}), 
passing to a subsequence of $t_n^j$, 
$$ |[e^{-it_n^j \partial_x^2} \phi^j](0)| \to 0, \quad n\to \infty$$
and thus 
$$ \frac{1}{2}\|\phi^j\|^{\frac{2(1-\sigma_c)}{\sigma_c}}_{L^2} \|\partial_x \phi^j\|_{L^2}^2 <A.$$
Since $A< M(\varphi_0)^{\frac{1-\sigma_c}{\sigma_c}} E(\varphi_0)$, $\phi^j$ satisfies the assumption of Lemma 
\ref{lem:ex_waveop}. Namely, there exists 
$\tilde{\phi^j} \in H^1$ such that 
$$\|\mathrm{NLS}(-t_n^j)\tilde{\phi}^j- e^{-it_n^j\partial_x^2} \phi^j\|_{H^1} \to 0, \quad n\to \infty$$
with 
$$ M(\tilde{\phi^j})=\|\phi^j\|_{L^2}^2, \quad E(\tilde{\phi^j})=\frac{1}{2} \|\partial_x \phi^j\|_{L^2}^2,$$

$$ \|\partial_x \mathrm{NLS}(t) \tilde{\phi^j}\|_{L^2} \|\tilde{\phi}^j\|_{L^2}^{\frac{1-\sigma_c}{\sigma_c}} 
< \|\varphi_0\|_{L^2}^{\frac{1-\sigma_c}{\sigma_c}} \|\partial_x \varphi_0\|_{L^2},$$
and thus
$$M(\tilde{\phi^j})^{\frac{1-\sigma_c}{\sigma_c}}E(\tilde{\phi^j})<A.$$
Therefore by the definition of threshold $A$, we have 
\begin{equation} \label{eq:4}
\|\mathrm{NLS}(t) \tilde{\phi}^j(0)\|_{L^q_{\R_t}} <+ \infty.  
\end{equation}

If the Case b), by the time continuity in $H^1_x$ norm of the linear flow, we know  
$$ e^{-i t_n^j \partial_x^2} \phi^j \to e^{-it_* \partial_x^2} \phi^j ~\mbox{in}~ H^1_x.$$
Thus it suffices to put $\tilde{\phi^j}:= \mathrm{NLS}(t_*)[e^{-it_*\partial_x^2}\phi^j]$. 
Then this $\tilde{\phi}^j$ again satisfies (\ref{eq:4}). 
To see this, note first that by the $H^1$ continuity of the flow, sending $n\to \infty$ in \eqref{E:methresh} gives 
$$M(e^{-it_*\partial_x^2} \phi^j)^{\frac{1-\sigma_c}{\sigma_c}} E(e^{-it_*\partial_x^2} \phi^j) \leq  A_j<A$$
By \eqref{E:dotnorm_expansion} applied for $\sigma_c=0$ and $\sigma_c=1$, and the assumption that $\eta_n(0)<1$ for every $n$, we obtain that 
$$\frac{\| \phi^j \|_{L_x^2}^{\frac{1-\sigma_c}{\sigma_c}} \|\partial_x \phi^j \|_{L_x^2}}{ \|\varphi_0 \|_{L_x^2}^{\frac{1-\sigma_c}{\sigma_c}} \| \partial_x \varphi_0 \|_{L_x^2}} < 1.$$
By the defining property of the threshold $A$, we have that the NLS flow with initial data $e^{-it_*\partial_x^2}\phi^j$ scatters, i.e.
$$\| \mathrm{NLS}(t) \tilde \phi^j(0) \|_{L_{\R_t}^q} = \| \mathrm{NLS}(t+t_*)e^{-it_*\partial_x^2}\phi^j(0) \|_{L_{\R_t}^q} < \infty.$$ 
\vspace{3mm}

Now replace $e^{-i t_n^j \partial_x^2} \phi^j$ by $\mathrm{NLS} (-t_n^j) \tilde{\phi}^j$ in (\ref{eq:1}), and
we have 
$$
\psi_{0,n} = \sum_{j=1}^M \mathrm{NLS}(-t_n^j) \tilde{\phi}^j +\tilde{w}_n^M,
$$
with 
$$\tilde{w}_n^M= w_n^M + \sum_{j=1}^M (e^{-it_n^j \partial_x^2} \phi^j- \mathrm{NLS}(-t_n^j) \tilde{\phi}^j).$$
Note that by Sobolev embedding and Proposition \ref{P:smoothing} (1),   
\begin{eqnarray*}
&&\|[e^{it\partial_x^2} \tilde{w}_n^M](0)\|_{L^q_{\R_t}} \\
&\le&
\|[e^{it\partial_x^2} w_n^M](0)\|_{L^q_{\R_t}} 
+ \sum_{j=1}^M \|[e^{it\partial_x^2}(-\mathrm{NLS}(-t_n^j) \tilde{\phi}^j +e^{-it_n^j \partial_x^2} \phi^j)](0)\|_{L^q_{\R_t}} \\
&\le& \|[e^{it\partial_x^2} w_n^M ](0)\|_{L^q_{\R_t}}  
+ \sum_{j=1}^M \|\mathrm{NLS}(-t_n^j) \tilde{\phi}^j -e^{-it_n^j \partial_x^2} \phi^j\|_{\dot{H}_x^{\sigma_c}}, \\
&\le & 
\|[e^{it\partial_x^2} w_n^M ](0)\|_{L^q_{\R_t}}  
+ \sum_{j=1}^M \|\mathrm{NLS}(-t_n^j) \tilde{\phi}^j -e^{-it_n^j \partial_x^2} \phi^j \|_{H^1_x}. 
\end{eqnarray*}
Thus we obtain, 
$$\lim_{M\to +\infty} [\lim_{n\to +\infty} \|[e^{i t \partial_x^2}\tilde{w}_n^M](0)\|_{L^q_{\R_t}} ]=0.$$
From this way of writing we might approximately see  
$$\mathrm{NLS}(t)\psi_{n,0} \approx \sum_{j=1}^M \mathrm{NLS}(t-t_n^j) \tilde{\phi}^j.$$
However, from (\ref{eq:4}), the RHS is finite in $L^q_{\R_t}$ norm, while the LHS cannot scatter by assumption, and so  
a contradiction could be deduced. We shall justify this argument by Proposition \ref{P:longtimeper}. 
\vspace{3mm}

Let $v^j(t):=\mathrm{NLS}(t) \tilde{\phi}^j,$ $\psi_n:=\mathrm{NLS}(t)\psi_{0,n},$ and $\tilde{\psi}_n=\sum_{j=1}^M v^j(t-t_n^j).$ 
Then, $\tilde{\psi}_n$ satisfies 
$$ i\partial_t \tilde{\psi}_n +\partial_x^2 \tilde{\psi}_n +\delta (|\tilde{\psi}_n|^{p-1} \tilde{\psi}_n +e_n)=0.$$
Here, 
$$ e_n:= - |\tilde{\psi}_n|^{p-1} \tilde{\psi}_n + \sum_{j=1}^M |v^j(t-t_n^j)|^{p-1} v^j(t-t_n^j).$$
We are going to show that
\begin{itemize}
 \item[1] there exists a large constant $A$ independent of $M$ satisfying the following property: 
 for any $M$ there is $n_0=n_0(M)$ such that if $n>n_0,$ $\|\tilde{\psi}_n (0,\cdot)\|_{L^q_{\R_t}} \le A$.
 \item[2] For each $M$ and $\varepsilon>0$ there exists $n_1=n_1(M, \varepsilon)$ such that for $n>n_1$, 
 $\|e_n\|_{L^{\tilde{q}}_{\R_t}} \le \varepsilon.$ 
 \end{itemize}
Remark that there exists $M_1=M_1(\varepsilon)$ such that for each $M>M_1$, there exists $n_2=n_2(M)$ such that 
if $n>n_2$, $\|[e^{it\partial_x^2} (\tilde{\psi}_n(0)-\psi_n(0))](0)\|_{L^q_{\R_t}} \le \varepsilon.$ Thus, if the above 1 and 2 hold, 
it follows 
from Proposition \ref{P:longtimeper} that for $n$ and $M$ sufficiently large, $\|\psi_n\|_{L^q_{\R_t}} < \infty,$ 
which gives a contradiction. Therefore it is enough to prove the above claims 1 and 2. 
First we prove the claim 1. Take $M_0$ large enough so that 
$$\|[e^{it\partial_x^2} w_n^{M_0}](0)\|_{L^q_{\R_t}} \le \delta_{\sd}/2.$$ 
Then, by Lemma \ref{L:lqbd}, for each $j >M_0$, we have $\|[e^{i(t-t_n^j)\partial_x^2} \phi^j](0)\|_{L^q_{\R_t}} \le \delta_{\sd}.$  
Thus by Lemma \ref{lem:ex_waveop} we obtain, for each $j>M_0$, and for large $n$,   
\begin{equation} \label{eq:3}
\|v^j(0,\cdot-t_n^j)\|_{L^q_{\R_t}} \le 2 \|[e^{i(t-t_n^j) \partial_x^2} \phi^j](0)\|_{L^q_{\R_t}}. 
\end{equation}
By Minkowski inequality (since $p>3$), 
\begin{eqnarray*}
&& \|\tilde{\psi_n}(0,\cdot)\|_{L^q_{\R_t}}^q \\
&\le& C_q \Big(\Big\|\sum_{j=1}^{M_0} v^j(0, \cdot-t_n^j)\Big\|_{L^q_{\R_t}}^q 
+\left\|\sum_{j=M_0+1}^{M} v^j(0, \cdot-t_n^j)\right\|_{L^q_{\R_t}}^q \Big)  \\
&\le& C_q \Big(\sum_{j=1}^{M_0} \|v^j(0, \cdot-t_n^j)\|_{L^q_{\R_t}}^2 +  \sum_{j=M_0+1}^{M} \|v^j(0, \cdot-t_n^j)\|_{L^q_{\R_t}}^2 \\
&& +\sum_{j \ne m, j,m=1}^{M_0}\| v^j(0, \cdot-t_n^j) v^m(0, \cdot-t_n^m)\|_{L^{q/2}_{\R_t}}^{q/2} \\
&& +\sum_{j \ne m, j,m=M_0+1}^{M}\| v^j(0, \cdot-t_n^j) v^m(0, \cdot-t_n^m)\|_{L^{q/2}_{\R_t}}^{q/2}
\Big) \\
&\le & C_q \Big(\sum_{j=1}^{M_0} \|v^j(0, \cdot-t_n^j)\|_{L^q_{\R_t}}^2 + \sum_{j=M_0+1}^{M} \|[e^{i(t-t_n^j) \partial_x^2} \phi^j](0) \|_{L^q_{\R_t}}^2 \\ 
&& +\sum_{j \ne m, j,m=1}^{M_0}\| v^j(0, \cdot-t_n^j) v^m(0, \cdot-t_n^m)\|_{L^{q/2}_{\R_t}}^{q/2} \\
&& +\sum_{j \ne m, j,m=M_0+1}^{M}\| v^j(0, \cdot-t_n^j) v^m(0, \cdot-t_n^m)\|_{L^{q/2}_{\R_t}}^{q/2}\Big)
\end{eqnarray*}
where we have used (\ref{eq:3}).
The last terms $\sum_{j \ne m} \|v^j v^m\|_{L^{q/2}_t}$ can be made small if $n$ is large (see the argument below for the claim 2).  
On the other hand, using (\ref{eq:1}), the same argument for (\ref{E:piecewise_expansion}) allows us to obtain  
\begin{eqnarray*}
|[e^{it\partial_x^2}\psi_{0,n}](0)|^q = 
\sum_{j=1}^{M} |[e^{i(t-t_n^j)\partial_x^2} \phi^j](0)|^q + |[e^{it\partial_x^2} w_n^M](0)|^q +o_n(1),  
\end{eqnarray*}
thus, integrating in time, 
\begin{eqnarray*}
\|[e^{it\partial_x^2}\psi_{0,n}](0)\|_{L^q_{\R_t}} &=& 
\sum_{j=1}^{M_0} \|[e^{i(t-t_n^j)\partial_x^2} \phi^j](0)\|_{L^q_{\R_t}} \\ 
&& \hspace{5mm}+\sum_{j=M_0+1}^{M} \|[e^{i(t-t_n^j)\partial_x^2} \phi^j](0)\|_{L^q_{\R_t}}
+ \|[e^{it\partial_x^2} w_n^M](0)\|_{L^q_{\R_t}} + o_n(1)
\end{eqnarray*}
which shows that $\sum_{j=M_0+1}^{M} \|e^{i(t-t_n^j)\partial_x^2} \phi^j \|_{L^q_{\R_t}}^2$ is bounded independently of $M$ if $n>n_0$
since $\|[e^{it\partial_x^2}\psi_{0,n}](0)\|_{L^q_{\R_t}} \le \|\psi_{0,n}\|_{\dot{H}^{\sigma_c}}$. 
Recall that $\|v^j(0, \cdot-t_n^j)\|_{L^q_{\R_t}}=\|\mathrm{NLS}(t) \tilde{\phi}^j (0) \|_{L^q_{\R_t}} <\infty$.
Therefore $\|\tilde{\psi_n}(0,\cdot)\|_{L^q_{\R_t}}^q$ is bounded independently of $M$ provided $n>n_0$.

We next prove the claim 2. We see that $e_n$ is estimated using H\"older inequality with $\frac{1}{\tilde{q}}=\frac{p-2}{q}+\frac{2}{q}$ as follows. 
\begin{eqnarray*}
&&\|e_n\|_{L^{\tilde{q}}_{\R_t}} \\
&\le& C_p \sum_{j=1}^M \Big(\|v^j\|_{L^q_{\R_t}}^{p-2}+ \Big\|\sum_{j=1}^M v^j\Big\|_{L^q_{\R_t}}^{p-2}\Big) 
 \|(v^1+\cdots+v^{j-1}+v^{j+1}+\cdots+v^M)v^j\|_{L_{\R_t}^{q/2}} 
\end{eqnarray*}
where we abbreviated $v^j(0,t-t_n^j)$ as $v^j$. Here, note that by (\ref{eq:4}), for any $\varepsilon>0$, there exists a large $R>0$ such that 
$$ \|\mathrm{NLS}(t-t_n^k) \tilde{\phi}^k (0)\|_{L^{q}(\{t:|t-t_n^k|>R\})} <\varepsilon.$$
Thus, taking large $n$ such that $|t_n^j-t_n^k| >2R$ with $j\ne k$ for such a $R>0$, we can estimate $\|v^j v^k\|_{L^{q/2}_{\R_t}}$ as follows:
\begin{eqnarray*}
\|v^j v^k \|_{L^{q/2}_{\R_t}} &\le& \|[\mathrm{NLS}(t-t_n^j) \tilde{\phi}^j] (0) [\mathrm{NLS}(t-t_n^k) \tilde{\phi}^k] (0)\|_{L^{q/2}_{\R_t}} \\
&\le& \|\mathrm{NLS}(t-t_n^j) \tilde{\phi}^j(0)\|_{L^{q}(\{t:|t-t_n^j|>R\})}\|\mathrm{NLS}(t-t_n^k) \tilde{\phi}^k (0)\|_{L_{\R_t}^{q}} \\
&&+\|\mathrm{NLS}(t-t_n^j) \tilde{\phi}^j (0)\|_{L_{\R_t}^{q}}\|\mathrm{NLS}(t-t_n^k) \tilde{\phi}^k (0)\|_{L^{q}(\{t:|t-t_n^k|>R\})} \\
&\le& C\varepsilon. 
\end{eqnarray*}
This shows that there exists $n_1$ such that the $L^{\tilde{q}}$ norm of $e_n$ is small if $n>n_1(M, \varepsilon)$. 
\vspace{3mm}

Now we consider Case 2. In this case, we have $M(\phi^1) \le 1$ and $\lim_{n\to\infty} E(e^{-it_n^1\partial_x^2} \phi^1) \le A.$ 
As in the Case 1, by the existence of wave operator, there is $\tilde{\phi}^1 \in H^1_x$ such that 
$$\|\mathrm{NLS}(-t_n^1) \tilde{\phi}^1 - e^{-it_n^1 \partial_x^2} \phi^1 \|_{H^1} \to 0, \quad n\to +\infty.$$
Put $$\tilde{w}_n^M := w_n^M -\mathrm{NLS}(-t_n^1) \tilde{\phi}^1 +e^{-it_n^1 \partial_x^2} \phi^1$$
Then we can write 
$$\psi_{0,n} = e^{-it_n^1 \partial_x^2} \phi^1 +w_n^M = \mathrm{NLS}(-t_n^1) \tilde{\phi}^1 +\tilde{w}_n^M,$$
with 
$$ \lim_{M\to \infty} \lim_{n\to \infty} \|[e^{it\partial_x^2} \tilde{w}_n^M](0)\|_{L^q_{\R_t}} =0.$$

Let $\psi_c$ be the solution to (\ref{E:pNLS}) with initial data $\psi_{c}(0)=\tilde{\phi}^1$. 
Now we claim that $\|\psi_c(0,\cdot)\|_{L^q_{\R_t}}=+\infty$ (and thus $M(\psi_c)^{\frac{1-\sigma_c}{\sigma_c}} E(\psi_c)=A$). We proceed as in the Case 1. 
Suppose $A:=\|\psi_c(0,\cdot)\|_{L^q_{\R_t}} <\infty.$ 
By definition, $\|\mathrm{NLS}(t)\tilde{\phi}^1(0)\|_{L^q_{\R_t}}=\|\psi_c(0,\cdot)\|_{L^q_{\R_t}} =A.$
For any shift $t'$, we can say $\|\mathrm{NLS}(t-t') \tilde{\phi}^1(0)\|_{L^q_{\R_t}}=
\|\mathrm{NLS}(t)\tilde{\phi}^1(0)\|_{L^q_{\R_t}},$ thus we take in particular $t'=t_n^1$ and operate 
$\mathrm{NLS}(t)$ to $\psi_{0,n}=\mathrm{NLS}(-t_n^1) \tilde{\phi}^1 +\tilde{w}_n^M.$ 
We apply the perturbation argument by 
Proposition \ref{P:longtimeper} to 
$$ \psi_n = \tilde{\psi}_n + \mathrm{NLS}(t) \tilde{w}_n^M,$$
with $\tilde{\psi}_n=\mathrm{NLS}(t-t_n^1) \tilde{\phi}^1$ 
and $\|\tilde{\psi_n}(0,\cdot)\|_{L^q_{\R_t}} = A<+\infty$. For $n$ and $M$ sufficiently large, 
we have 
$$\|[e^{it\partial^2_x} (\psi_n(0)-  \tilde{\psi}_n(0))] (0)\|_{L^q_{\R_t}}=\|[e^{it\partial_x} \tilde{w}_n^M](0)\|_{L^q_{\R_t}} 
\le \epsilon_0,$$
and also the $L^{\tilde{q}}_t$ norm of the corresponding error term is estimated by $\epsilon_0$,  
where $\epsilon_0=\epsilon_0(A)$ is obtained in Proposition \ref{P:longtimeper}.  
Then, by Proposition \ref{P:longtimeper}, we have $\|\psi_n(0,\cdot)\|_{L^q_{\R_t}} <\infty$, and this is a contradiction to 
non scattering assumption on $\psi_n.$
\vspace{3mm}

On the other hand, the proof of Lemma 5.6 in \cite{HR} allows us to have also, 

\begin{lemma} \label{unif_localization} 
Suppose $\{\psi(t,x), t \ge 0 \}$ is precompact in $H^1_x$. Then 
for any $\varepsilon>0$, there exists $R_{\varepsilon}>0$ such that 
$$\sup_{t \ge 0} \int_{|x| \ge R_{\varepsilon}} (|\psi(x,t)|^2 + |\partial_x \psi(t,x)|^2) dx \le \varepsilon. $$ 
\end{lemma}
\vspace{3mm}

Using this Lemma and the local viriel identity (\ref{E:local-virial}), we conclude the following proposition.

\begin{proposition} \label{P:localization} Let $p>3$. 
Assume $\psi_0 \in H^1$ satisfies (\ref{E:ME}) and $\eta(0)<1.$ 
Let $\psi(t,x)$ be the global solution to (\ref{E:pNLS}) with the initial data $\psi_0$ satisfying the precompactness:
for any $\varepsilon>0$, there exists $R_{\varepsilon}>0$ such that 
\begin{equation} \label{eq:localization}
\int_{|x| \ge R_{\varepsilon}}(|\psi(x,t)|^2+|\partial_x \psi(x,t)|^2) dx \le \varepsilon, \quad \mbox{for all} \quad t \ge 0.  
\end{equation}
Then $\psi_0 \equiv 0$.
\end{proposition}

\proof Take $a(x)$ in the localized virial (\ref{E:local-virial}), 
as, for $R>0$ (which will be determined later), and for all $x\in \mathbb{R}$, 
$$ a(x) =R^2 \chi \Big(\frac{|x|}{R}\Big),$$ 
where $\chi \in C_0^{\infty}(\mathbb{R}^{+})$, $\chi(r)=r^2$ for $r \le 1$, and $\chi(r)=0$ for $r \ge 2$. 
Put $z_R(t):=\int_{\R} a(x) |\psi|^2 dx$, then we have 
\begin{equation*}
z'_R(t)=-2R \Im \int_{\R} \chi'\Big(\frac{|x|}{R}\Big) \partial_x \psi \overline{\psi} dx, 
\end{equation*}
and
\begin{eqnarray} \label{specialcut} \nonumber
z_R''(t)
&=&  
8 \int_{|x| \le R} |\partial_x \psi|^2 dx + 4 \int_{R < |x|<2R} \chi'' \Big(\frac{|x|}{R} \Big) |\partial_x \psi|^2 dx \\ \nonumber
&&
-\frac{1}{R^2}\int_{R < |x| <2R} \chi^{(4)} \Big(\frac{|x|}{R}\Big) |\psi|^2 dx
-4 |\psi(0)|^{p+1}  \\  \nonumber
& \ge & 
2\{4 \int_{|x| \le R} |\partial_x \psi|^2 dx-2|\psi(0)|^{p+1} \} -C_0 \int_{R < |x| <2R} (|\partial_x \psi|^2 + \frac{1}{R^2}|\psi|^2) dx \\
& \ge & 
2\{4 \int_{|x| \le R} |\partial_x \psi|^2 dx-2|\psi(0)|^{p+1} \} -C_0 \int_{R < |x|} (|\partial_x \psi|^2 + \frac{1}{R^2}|\psi|^2) dx
\end{eqnarray}
with a constant $C_0=C_0 (\|\chi''\|_{L^{\infty}}, \|\chi^{(4)}\|_{L^{\infty}})$ uniform in $R$.  

Take $0<\delta<1$ such that 
$$ M(\psi_0)^{\frac{1-\sigma_c}{\sigma_c}}E(\psi_0)\le (1-\delta)M(\varphi_0)^{\frac{1-\sigma_c}{\sigma_c}}E(\varphi_0), $$
then by (\ref{E:LB}), there exists $c_{\delta}>0$ such that for any $t\in \R$ 
\begin{equation} \label{eq:pre}
4 \int_{|x| \le R} |\partial_x \psi|^2 dx-2|\psi(0)|^{p+1} \ge c_{\delta} \|\partial_x \psi_0\|_{L^2}^2-4 \int_{|x| >R} |\partial_x \psi|^2 dx.
\end{equation}
Now, we choose  
$\varepsilon=\frac{c_\delta}{8+C_0}\|\partial_x \psi_0\|_{L^2}^2$ in (\ref{eq:localization}), 
then for sufficiently large $R_1 >\max\{1, R_{\varepsilon}\}$,  
$$\int_{|x| >R_1} \Big(|\partial_x \psi|^2 +\frac{1}{R_1^2}|\psi|^2 \Big) dx 
\le \int_{|x| >R_1} \Big( |\partial_x \psi|^2 +|\psi|^2 \Big) dx \le \varepsilon
=\frac{c_\delta}{8+C_0}\|\partial_x \psi_0\|_{L^2}^2.$$
Thus, by the choice of $R=R_1$, we have $(\ref{eq:pre}) \ge c_{\delta} \|\partial_x \psi_0 \|_{L^2}^2-4 \varepsilon$ 
and so 
$$z''_{R_1}(t) \ge c_{\delta} \|\partial_x \psi_0\|_{L^2}^2. $$
Integration in time then implies 
$$
z'_{R_1}(t) -z'_{R_1}(0) \ge c_{\delta} t\|\partial_x \psi_0\|_{L^2}^2.
$$
On the other hand, 
$$|z'_{R_1}(t) -z'_{R_1}(0)| \le C R_1$$
where $C$ depends on $p, \|\psi_0\|_{L^2}$, and $\|\partial_x \psi_0\|_{L^2}.$
This is absurd except the case $\psi_0 \equiv 0$. 
\hfill\qed
\vspace{3mm}

Finally we complete our arguments with  
\begin{proposition}
\begin{equation*}
K=\{\psi_c(t), t \ge 0 \} \subset H^1_x 
\end{equation*}
with $\psi_c$ obtained above as the minimal non scattering solution, is precompact in $H^1_x$.
\end{proposition}
The proof for this proposition is similar to the proof for the existence of $\psi_c$, and we omit it. 
We apply Proposition \ref{P:localization} to $\psi_c$, and we have $\psi_c (0) \equiv 0$, which contradicts 
the fact that $\|\psi_c(0,\cdot)\|_{L^q_{\R_t}}=+\infty$. This concludes the statement of Theorem \ref{T:main}. 
\hfill\qed
\vspace{3mm}

\noindent
{\bf Acknowledgment} 
The work described in this paper is a result of a collaboration made possible by the IMA's 
annual program workshop "Mathematical and Physical Models of Nonlinear Optics." 
This work was supported by JSPS KAKENHI Grant Number 15K04944.


\begin{thebibliography}{00}

\bibitem{ACCT} 
R. Adami, R. Carlone, M. Correggi, L. Tentarelli, \emph{ Blow-up for the pointwise NLS in dimension two: absence of critical power}, preprint arXiv:1808.10343 (2018).

\bibitem{ADFT} R. Adami, G. Dell'Antonio, R. Figari, A. Teta, 
\emph{ The Cauchy Problem for the Schr\"odinger Equation in Dimension Three with Concentrated Nonlinearity}, Ann. Inst. H. Poincar\'e (C) An. Nonlin. 20 (2003), no. 3, pp. 477--500.

\bibitem{AT}  
R. Adami and A. Teta,\emph{ A class of nonlinear Schr\"odinger equations with concentrated nonlinearity}, J. Funct. Anal. 180 (2001), no. 1, pp. 148--175. 

\bibitem{BV} V. Banica and N. Visciglia, 
\emph{Scattering for NLS with a delta potential}, J. Differential Equations. 260 (2016), no.5, pp. 4410–4439.

\bibitem{BKKS} V.S. Buslaev, A.I. Komech, E. A. Kopylova, D. Stuart,  
\emph{ On asymptotic stability of solitary waves in Schr\"odinger equation coupled to nonlinear oscillator}, Commun. in PDE 33 (2008), no. 4, pp. 669--705.

\bibitem{CCT} R. Carlone, M. Correggi, L. Tentarelli, \emph{ Well-posedness of the two-dimensional nonlinear  equation with concentrated nonlinearity}, 
Ann. IHP (c) An. Nonlin. 36 2019 no. 1, pp. 257--294.

\bibitem{CFNT}
C. Cacciapuoti, D. Finco, D. Noja and A. Teta, \emph{The NLS equation in dimension one with spatially concentrated nonlinearities: the pointlike limit}, 
Lett. Math. Phys. 104 (2014), no. 12, pp. 1557--1570. 

\bibitem{CW} F.M. Christ and M. I. Weinstein, \emph{Dispersion of small amplitude solutions of the generalized Korteweg-de Vries equation}, 
J. Funct. Anal. 100 (1991) pp. 87-109.

\bibitem{DHR}
T. Duyckaerts, J. Holmer, and S. Roudenko, \emph{Scattering for the non-radial 3D cubic nonlinear Schr\"odinger equation}, 
Math. Res. Lett. 15 (2008), no. 6, pp. 1233--1250.

%\bibitem{holmer} The initial-boundary value problem for the 1D
%nonlinear Schr\"odinger equation on the half-line

\bibitem{HR}
J. Holmer and S. Roudenko, \emph{A sharp condition for scattering of the radial 3D cubic nonlinear Schr\"odinger equation}, 
Comm. Math. Phys. 282 (2008), no. 2, pp. 435--467.

\bibitem{HL1}
J. Holmer, C. Liu, \emph{Blow-up for the 1D nonlinear Schr\"odinger equation with point nonlinearity I: Basic theory}, preprint. 

\bibitem{HL2} 
J. Holmer and C. Liu, \emph{Blow-up for the 1D nonlinear Schr\"odinger equation with point nonlinearity II: Supercritical blowup profiles}, preprint.

\bibitem{HTMG}
D. Hennig, G.P Tsironis, M.I. Molina, and H. Gabriel, \emph{A nonlinear quasiperiodic Kronig-
Penney model}, Physics Letters A 190 (1994) 

\bibitem{KKS}, A.I. Komech, E. A. Kopylova, D. Stuart,  
\emph{On asymptotic stability of solitons in a nonlinear Schr\"odinger equation}, Commun. in Pure App. An. 11 (2012), no. 3, pp. 1063--1079.

\bibitem{KM}
C. E. Kenig, F. Merle, \emph{Global well-posedness, scattering and blow-up for the energy-critical, focusing, non-linear Schr\"odinger equation in the radial case}, 
Invent. Math. 166 (2006), no. 3, pp. 645--675. 

\bibitem{MB}
M. I. Molina and C. A. Bustamante, \emph{The attractive nonlinear delta-function potential},
arXiv:physics/0102053 [physics.ed-ph] (2001).

\bibitem{N} K. Nakanishi, \emph{Energy scattering for nonlinear Klein-Gordon and Schr\"odinger equations in
spatial dimensions 1 and 2}, J. Funct. Anal. 169 (1999) 201-225. 

\bibitem{S} E. M. Stein, \emph{Harmonic Analysis: real-variable methods, orthogonality, and oscillatory integrals, 
with the assistance of Timothy S. Murphy},
Princeton University Press, Princeton, New Jersey (1993). 
\end{thebibliography}
\end{document}